\documentclass[10pt]{article}
\usepackage{amsmath,amsfonts}
\usepackage{array}
\usepackage{textcomp}
\usepackage{bbding} 
\usepackage{geometry}
\usepackage{textcomp}
\usepackage{stfloats}
\usepackage{url}
\usepackage{verbatim}
\usepackage{cite}
\usepackage{amsmath,amsfonts}
\usepackage{booktabs}
\usepackage{color}
\usepackage{graphicx}
\usepackage{subcaption} 
\usepackage{float}
\usepackage{algorithm}
\usepackage{algpseudocode}

\usepackage{amssymb}
\usepackage{amsthm}
\usepackage{multirow} 
\usepackage{makecell}
\usepackage{diagbox}
\usepackage{threeparttable}
\usepackage{CJKutf8}

\usepackage{changes}
\usepackage{comment}

\newtheorem{theorem}{Theorem}
\newtheorem{lemma}{Lemma}

\newtheorem{eg}{Example}
\newtheorem{definition}{Definition}

\providecommand{\U}[1]{\protect\rule{.1in}{.1in}}


\usepackage{ulem}

\usepackage[hidelinks,colorlinks=true,linkcolor=blue,citecolor=blue]{hyperref}
\usepackage{tikz}
\definecolor{lime}{HTML}{A6CE39}
\DeclareRobustCommand{\orcidicon}{%
	\begin{tikzpicture}
	\draw[lime, fill=lime] (0,0) 
	circle [radius=0.16] 
	node[white] {{\fontfamily{qag}\selectfont \tiny ID}};	\draw[white, fill=white] (-0.0625,0.095) 
	circle [radius=0.007];	\end{tikzpicture}
	\hspace{-2mm}}
\foreach \x in {A, ..., Z}{%
	\expandafter\xdef\csname orcid\x\endcsname{\noexpand\href{https://orcid.org/\csname orcidauthor\x\endcsname}{\noexpand\orcidicon}}
	}

\title{Floorplanning with I/O Assignment via Feasibility-Seeking and Superiorization Methods}

\author{Shan Yu, 
Yair Censor, 
and Guojie Luo
\thanks{
This article is an extended version of our conference report~\cite{yu2023} presented at the International Symposium of EDA, ISEDA-2023, Nanjing, China, May 8-11, 2023. This research is supported by National Natural Science Foundation of China (NSFC) under Grant No. 11961141007 and by the Israeli Science Foundation (ISF) under Grant No. 2874/19 within the NSFC-ISF joint research program. The work of Y.C. is also supported by the U.S. National Institutes of Health Grant No. R01CA266467 and by the Cooperation Program in Cancer Research of the German Cancer Research Center (DKFZ) and the Israeli Ministry of Innovation, Science and Technology (MOST). Additionally, the work of G.L. is supported by National Key R\&D Program of China under Grant No. 2022YFB4500500. \textit{Corresponding Author: Guojie Luo.}
We are grateful to our colleagues Prof. Tie Zhou and Prof. Jiansheng Yang at the Peking University for their helpful comments.
We thank Xinming Wei, Bizhao Shi and Sunan Zou at the Peking University for their helpful suggestions throughout the preparation of this manuscript.
Shan Yu is with the Department of Information and Computational Sciences, School of Mathematical Sciences, Peking University, Beijing, China. Corresponding email: yu-shan@pku.edu.cn.
Yair Censor is with the Department of Mathematics, University of Haifa, Haifa, Israel. Corresponding email: yair@math.haifa.ac.il.
Guojie Luo is with the National Key Laboratory for Multimedia Information Processing, School of Computer Science, Peking University, Beijing, China. Dr. Luo is also with the Center for Energy-efficient Computing and Applications, Peking University, Beijing, China. Corresponding email: gluo@pku.edu.cn.
}
}
\begin{document}
\begin{CJK*}{UTF8}{gbsn}
\maketitle

\begin{abstract}
The feasibility-seeking approach offers a systematic framework for managing and resolving intricate constraints in continuous problems, making it a promising avenue to explore in the context of floorplanning problems with increasingly heterogeneous constraints.
The classic legality constraints can be expressed as the union of convex sets. However, conventional projection-based algorithms for feasibility-seeking do not guarantee convergence in such situations, which are also heavily influenced by the initialization. We present a quantitative property about the choice of the initial point that helps good initialization and analyze the occurrence of the oscillation phenomena for bad initialization. In implementation, we introduce a resetting strategy aimed at effectively reducing the problem of algorithmic divergence in the projection-based method used for the feasibility-seeking formulation. Furthermore, we introduce the novel application of the superiorization method (SM) to floorplanning, which bridges the gap between feasibility-seeking and constrained optimization. The SM employs perturbations to steer the iterations of the feasibility-seeking algorithm towards feasible solutions with reduced (not necessarily minimal) total wirelength. Notably, the proposed algorithmic flow is adaptable to handle various constraints and variations of floorplanning problems, such as those involving I/O assignment. 

To evaluate the performance of Per-RMAP, we conduct comprehensive experiments on the MCNC benchmarks and GSRC benchmarks. The results demonstrate that we can obtain legal floorplanning results 166$\times$ faster than the branch-and-bound (B\&B) method while incurring only a 5\% wirelength increase compared to the optimal results. Furthermore, we evaluate the effectiveness of the algorithmic flow that considers the I/O assignment constraints, which achieves an 6\% improvement in wirelength. 
Besides, considering the soft modules with a larger feasible solution space, we obtain 15\% improved runtime compared with PeF, the state-of-the-art analytical method. Moreover, we compared our method with Parquet-4 and Fast-SA on GSRC benchmarks which include larger-scale instances. The results highlight the ability of our approach to maintain a balance between floorplanning quality and efficiency.
\end{abstract}


\section{Introduction}
In the modern very-large-scale integration (VLSI) electronic design automation (EDA) flow, floorplanning is a critical stage in the physical design with significant impacts on the quality of downstream stages. With the functional modules  (e.g., intellectual property cores, embedded memories, and clusters of standard cells) and the connecting nets, the task of floorplanning~\cite{Kahng2011} is to find a legal placement within a fixed-outline rectangular region, where the basic constraint is non-overlapping between any two modules and the optimization objective is usually minimizing the total wirelength of nets. This problem is known to be challenging~\cite{Adya2003}, and its complexity is further compounded by the need to consider various factors to enhance the effectiveness of integrated circuits. 

Prior floorplanning studies can be broadly divided into discrete meta-heuristic methods, discrete exact methods (e.g., branch-and-bound (B\&B) methods), continuous analytical methods, and learning-based methods.

Discrete meta-heuristic methods use discrete representation and heuristics but struggle with diverse constraints. The reason is that complex heuristics accumulation would result in the unpredictability of tools~\cite{kahng2018}. Discrete exact methods systematically search solutions in a discrete space with pruning strategies to reduce computational effort. However, they still require significant time to explore the vast search space. Analytical methods formulate floorplanning as a constrained optimization problem. With constraints added as a penalty to the objectives, it is often converted to an unconstrained programming problem, leading to challenges in balancing constraints and objectives. Learning-based methods, which utilize techniques such as reinforcement learning and graph neural network, focus on the netlist representation and feature extraction, with limited exploration of multi-constraint problems.

In summary, two challenges still exist for current floorplanning methods. 
Firstly, aiming for the optimal solution is not reasonable because the resolving process requires needless expenditure of time, energy, and resources, while the favorable final results are still not guaranteed after the lengthy EDA flow.
Secondly, growing diverse constraints add great complexity to the design of the methods. These constraints include temperature management~\cite{ma21}, congestion alleviation~\cite{Ahn2011}, input/output (I/O) assignment~\cite{yu2023}, voltage island~\cite{Lin2014}, blockage regions~\cite{funke2016} and so on. Moreover, the continual advancements in package technologies and the increasing emphasis on cross-stage constraints~\cite{Kahng21} are leading to a significant growth in the number and complexity of constraints involved in floorplanning.

To overcome these challenges, we model the floorplanning as a \textbf{feasibility-seeking problem (FSP)}, wherein constraints are established through a finite family of nonempty, frequently closed constraint sets, aiming to locate a point that satisfies their collective intersection. Compared to previous approaches, the focus is shifted to the constraints rather than objectives, in the spirit of~\cite{Yair2023}, which makes the scalability to diverse constraints possible.

Support for such an approach is Simon's concept of ``satisfice''~\cite{simon1956}. Instead of adopting an approach that seeks optimal solutions within a simplified model, it is better to adopt an approach that seeks satisfactory solutions under more realistic and complex constraints. We adopt the philosophy of ``satisficing'' rather than optimizing to solve floorplanning with growing diverse constraints.

In the field of FSPs, the \textbf{method of alternating projection (MAP)}~\cite{von1951}\cite{Escalante2011} is a widely used method. It is flexible, computationally efficient, and of theoretical interest. It involves iterative projections onto individual sets in a sequential manner or in another ordering regime. With proper construction of individual constraints, it is fast, easy to implement, and enables to tackle complex constraints sets.
However, the constraints of the floorplanning problem are non-convex and the initial convergence behavior is affected by the initial point selection, exhibiting certain phenomena when unable to find a feasible solution. This motivates us to design a generalized MAP called the resettable method of alternating projection (RMAP).

To enhance the performance of a given FSP-based algorithm, the \textbf{superiorization method (SM)} offers a viable approach for obtaining superior (not necessarily optimal) feasible solutions. This method employs perturbations to guide a feasibility-seeking algorithm towards a feasible output while simultaneously improving a user-chosen objective function.
This article designs an SM called the perturbed resettable method of alternating projection (Per-RMAP) to further improve the feasible solution in terms of the total wirelength.

In this article, our contributions are summarized as follows:

\begin{itemize}
    \item We formulate floorplanning as an FSP, which promises a balance between solution quality and efficiency, as we prioritize feasible-seeking rather than optimizing. 
    \item We analyze the property and challenge of applying MAP to floorplanning, focusing on its initial convergence behavior. We prove the local convergence of MAP and analyze the occurrence of oscillations, offering valuable insights for the algorithm design.
    \item We propose our global floorplanner Per-RMAP to enhance MAP. We design a resetting strategy to improve the initial convergence behavior of the MAP. We also add perturbation steps to the original FSP algorithm to reduce the total wirelength. 
    \item Experiments on MCNC benchmark demonstrate that Per-RMAP could address complex design constraints. Compared to B\&B, it achieves a significant speedup of 166$\times$ with a merely 5\% increase in wirelength, and decreases wirelength by 6\% with a $\sim$100$\times$ speedup with I/O assignment constraints.\footnote{The B\&B method for comparison does not consider I/O assignment. The search space and pruning strategies have to be completely re-defined and re-implemented when considering extra variables or constraints in the B\&B method. In contrast, the FSP method readily adapts to a new formulation.} For soft-module floorplanning, our method reduces 15\% runtime compared with PeF, the state-of-the-art analytical method.
\end{itemize}

In the article we strengthen the results of the conference report~\cite{yu2023}. We analyze the property of the constraints sets and prove the local convergence of MAP, which offers the theoretical support and motivation of the resetting strategy design. Furthermore, we enhance the efficiency of Per-RMAP and conduct more comprehensive experiments to fully demonstrate the superior performance of Per-RMAP.

The remainder of the paper is organized as follows: Sec.~\ref{sec:background} introduces the preliminaries and related works. Sec.~\ref{sec:formulation} gives the formulation of floorplanning as FSP. Sec.~\ref{sec:MAP-property-and-challenges} presents features of applying MAP to floorplanning. Sec.~\ref{sec:PerRMAP} describes the proposed Per-RMAP algorithm. Sec.~\ref{sec:experiment} presents experimental results, followed by concluding comments in Sec.~\ref{sec:conclusion} and supplementary materials in Appendix~\ref{sec:example-n5} and \ref{sec:more-local-convergence}.
\section{Background}\label{sec:background}
In this section, we introduce the FSP, MAP, and SM briefly and then review some related works.

\subsection{Preliminaries}
\textbf{FSP} is a modeling process that represents the problem by a system of constraints. Given $M$ subsets $C_{i},i=1,2,\ldots,M$, of the N-dimensional Euclidean space $\mathbb{R}^{N}$, the FSP is to find a point in the intersection of these sets, i.e.,
\begin{align}
    \text{find }z\in \cap_{i=1}^M C_{i}. \label{eq:fsp}
\end{align}

Given a set $C\subset\mathbb{R}^{N}$, the (metric) projection of a point $z\in\mathbb{R}^{N}$ onto it is the set-valued mapping $\mathcal{P}_{C}$:
\begin{equation}
\mathcal{P}_{C}(z) = \left\{c \in C\ \vert\ \|z-c\| = \mathrm{d}(z,C)\right\},
\end{equation}
where $\|\cdot\|$ is the Euclidean norm, and $\mathrm{d}(z,C)=\mathop{\min}\left\{\|z-c\|\ \vert\ c\in C\right\}$. When $C$ is convex, $\mathcal{P}_{C}(z)$ is a singleton. When $C$ is non-convex, $\mathcal{P}_{C}(z)$ may contain multiple points, which is the situation in our case, as described below.

\textbf{MAP}~\cite{von1951}\cite{Escalante2011} is a widely used algorithm to tackle the FSP. It uses sequential projections iteratively onto the individual sets of the family of constraints in a predetermined order. 

The MAP iterative process for the FSP (\ref{eq:fsp}) is:
\begin{align}
\begin{split}
    (\forall n \in \mathbb{N}) \quad z^{n+1}
    & \in z^{n}+\lambda_{n} \left(\mathcal{P}_{t_{n}}\left(z^{n}\right)-z^{n}\right),\\
    & = \mathcal{T}_{t_{n}}(z^{n};\lambda_n),\label{eq:relaxed-projection}
\end{split}
\end{align}
where $C_{t_n}$ is the constraint for projection in the $n$-th iteration, point $z^{n+1}$ is obtained by choosing a projection\footnote{In the sequel, for any set $C_t$ with a subscript, we abbreviate the projection operator $\mathcal{P}_{C_t}$ by $\mathcal{P}_t$.} point from set $\mathcal{P}_{t_n}(z^n)$ according to various strategies (see Sec.~\ref{subsubsec:rmap}), $\mathcal{T}_{t_n}$ is the notation for relaxed projection with relaxation parameter $\lambda_n \in (0,2)$, and $t_{n}$ goes through all constraint indices in one sweep, i.e., $\left\{t_{kM+1},t_{kM+2},\ldots,t_{(k+1)M}\right\} = \left\{1,2,\ldots, M\right\}$. See, e.g,~\cite{Yair1997}.

\subsection{Related Work}\label{subsec:related-work}
\textbf{Floorplanning:} The methods can be broadly divided into discrete meta-heuristic methods, discrete exact methods (e.g., branch-and-bound (B\&B) methods), continuous analytical methods, and learning-based methods.

Discrete methods encompass two primary categories: meta-heuristic methods and exact methods. Meta-heuristic methods are employed to seek sub-optimal solutions, while exact methods strive to achieve optimality. They typically employ a representation, such as sequence pair~\cite{Murata1996}, $B^{*}$-tree~\cite{Chang2000}\cite{Anand2012} or constraint graph~\cite{funke2016}, to encode the relative positions of modules. Subsequently, they search for a solution within the representation space and decode it to obtain the corresponding floorplan. Anand et al.~\cite{Anand2012} utilize a $B^{*}$-tree representation and propose a heuristic based on simulated annealing (SA) to find floorplans with minimized dead space. Funke et al.~\cite{funke2016} employ a constraint graph and a B\&B method to achieve wirelength-optimal floorplans while considering specific sets of blocked regions. The multilevel heuristic reduces solution space by clustering or partitioning, handling large-scale problems. Chen et al.~\cite{chen2008} present IMF, starting with min-cut partitioning and merging via SA with $B^{*}$-tree. Yan et al.~\cite{yan2010defer} propose DeFer based on deferred decision-making, using the generalized slicing tree and an enumerate packing technique to defer the decision on the slicing structure, followed by block swapping and mirroring to enhance wirelength. Ji et al.~\cite{ji2021quasi} introduce QinFer, recursively bipartitioning the circuit into leaf subcircuits for distributed floorplanning, optimizing wirelength, employing a Quasi-Newton method to reduce overlap, and enhancing robustness with a refined distribution algorithm.

Analytical methods model floorplanning as a continuous optimization problem with quadratic~\cite{Zhan2006} or nonlinear~\cite{Lin2014}\cite{Li2022} objective functions. Typically, these methods involve a global floorplanning step to obtain approximate positions for all modules, followed by a legalization step to eliminate overlaps and obtain the precise positions of the modules. Zhan et al.~\cite{Zhan2006} use the bell-shaped function to smooth the density function in global floorplanning and generate a corresponding sequence pair and obtain a floorplan without any overlap between modules for legalization. Lin et al.~\cite{Lin2014} solve floorplanning with voltage-island constraints by a two-stage method. A density-driven nonlinear analytical method is firstly applied in global floorplanning, then a slicing tree is built to record the global distribution result and to determine the final feasible result in the legalization step. Lin et al.~\cite{lin2018fast} propose a multi-level thermal-aware floorplanning method which integrates the analytical thermal model into the non-linear placement model, enabling approximate temperature and minimizing wirelength at the same time. Li et al.~\cite{Li2022} present an analytical method based on Poisson's equation in global floorplanning, and utilize a constraint graph-based legalization approach for floorplanning with soft modules.

Recently learning-based algorithms, especially reinforcement learning and graph neural network, have gained great popularity. GoodFloorplan~\cite{Xu2022} combines graph convolutional network and reinforcement learning to explore the design space effectively. Liu et al~\cite{Liu2022} employ graph attention to learn an optimized mapping between circuit connectivity and physical wirelength, and produce a chip floorplan using efficient model inference.

I/O pin assignment, impacting the total wirelength in the order of 5\%~\cite{bandeira2020fast}, significantly influences the circuit performance. Several works incorporate it in floorplanning or placement. I/O pin assignment is to find locations of I/O pins along the boundaries of the chip. Modern analytical methods iteratively handle cell placement and I/O pin assignment, using I/O pin locations as a starting set of ``anchors'' to placement~\cite{Kim2012}\cite{Viswanathan2005}. Westra and Groeneveld~\cite{westra2005towards} extend the quadratic placement formulation to make concurrent global placement of cells and I/O pins. Banderia et al~\cite{bandeira2020fast} proposed a Hungarian matching-based heuristic, I/O Placer, which adopts a divide-and-conquer strategy for fast and scalable I/O assignment. This work is part of the OpenROAD project~\cite{ajayi2019openroad}.

\textbf{MAP:}
This is a method for handling the FSP. For the convex case, the method is also referred to as ``projection onto convex sets'' (POCS)~\cite{Bauschke2011}, and exhibits convergence~\cite{Yair1997} when the sets have a nonempty intersection, or cyclic convergence~\cite{Gubin1967} with empty intersection of sets under certain configurations.

For the FSP with non-convex sets, recent literature studies the convergence of the so-called Douglas–Rachford (DR) method, see, e.g., \cite{Bauschke2024}. For FSP with two closed sets, Bauschke et al.~\cite{Bauschke2014} prove local convergence of the DR method to a fixed point when each of the sets is a finite union of convex sets and observe the cyclic phenomenon under more general non-convex sets. Artacho et al.~\cite{Artacho2016} establish global convergence and describe the global behavior of the DR method for the problem which involves finding a point in the intersection of a half-space and a non-convex set with well-quasi-ordering property or a property weaker than compactness. However, DR is suited for the two-set case and is not simple to extend to problems under diverse constraints in practice.

There is limited theoretical work available on the general MAP for the non-convex case, particularly when dealing with union convex sets.  
St{\'e}phane et al.~\cite{chretien1996cyclic}\cite{chretien2020} prove the convergence of MAP under some strict conditions.

\textbf{SM:}
A recent tutorial~\cite{Yair2023} on the SM
provides valuable insights into the SM and references numerous recent works and sources on this subject. Blake Schultze et al.~\cite{yair2020} apply the total
variation superiorization to image reconstruction in proton computed
tomography to improve the result. Since its inception in 2007, the SM has evolved and gained ground, as can be seen from the, compiled and continuously updated, bibliography at: http://math.haifa.ac.il/yair/bib-superiorization-censor.html\#top. Inspired by their approach, we propose
here a wirelength-superiorized FSP algorithm for floorplanning.

\section{Formulation of Floorplanning as an FSP}\label{sec:formulation}
In this section, we introduce the FSP-based formulation for the floorplanning problem. Firstly, we present the notations. Consider a floorplanning problem with $N_m$ functional modules from the module set $\mathcal{M}$, a set of pins $\mathcal{S}$, and a set of nets $\mathcal{E}$ of wire connections among the corresponding pins. Among the pins in $\mathcal{S}$, the $N_{io}$ I/O pins (net terminals) are within $\mathcal{S}_{io}$, where $\mathcal{S}_{io} \subset\mathcal{S}$ and the given floorplanning region is a 2D rectangle from $(0, 0)$ to $(W, H)$.
For each module $m_i \in \mathcal{M}$, $w_i$ and $h_i$ denote its width and height, respectively. Besides, for the function modules and I/O pins, their coordinates are recorded in $z = (x, y) \in \mathbb{R}^{2N},  N:=N_m + N_{io}$, 
where
\begin{align}
\begin{split}
    x &=\left(x_{1},x_{2},\ldots,x_{N}\right),\\
    y &=\left(y_{1},y_{2},\ldots,y_{N}\right),
\end{split}
\end{align}
where $x,y\in\mathbb{R}^{N}$ are the $x$-coordinates and $y$-coordinates of
modules from $\mathcal{M}$ and I/O pins from $\mathcal{S}_{io}$ respectively. 
This stacking establishes an injective linear mapping from
the coordinates of modules and I/O pins to $\mathbb{R}^{2N}$. Thus,
we have two representations for the coordinates of modules from
$\mathcal{M}$ and pins from $\mathcal{S}_{io}$. The representation in $\mathbb{R}^{2N}$ is used for establishing the FSP-based formulation and algorithm, while the 2-dimensional representation in the form $(x_{i},y_{i})$ is used for implementation. Here we specify $(x_{i},y_{i})$ of the module $m_i$ as the location of its bottom left corner.

With these notations, we further establish the three common conditions in the floorplanning as an FSP.

\noindent\textbf{Condition 1: Boundary.} Every module is within the given floorplanning region. Therefore, for $1 \leq i \leq N_m$, let
\begin{align}
\begin{split}
    B_{i}^{x}(z) & :=\{z\in\mathbb{R}^{2N}|\,0\leq x_{i}\leq W-w_{i}\},\\
    B_{i}^{y}(z) & :=\{z\in\mathbb{R}^{2N}|\,0\leq y_{i}\leq H-h_{i}\}.\label{eq:B_i^x}
\end{split}
\end{align}
If both the module $m_{i}$ and the module $m_{j}$ (where $1\le i,j\le N_{m}$ and $i\neq j$) fall in the floorplanning region, the following must hold, 
\begin{align}
z\in B_{i,j} & :=B_{i}^{x}\cap B_{i}^{y}\cap B_{j}^{x}\cap B_{j}^{y}.\label{box-constraints}
\end{align}

\noindent\textbf{Condition 2: Non-overlap.} Every pair of modules should have no overlap. Therefore, for $1\le i,j\le N_{m}$ and
$i\neq j$ is equivalent to one of the following four constraints
\begin{align}
\begin{split}
z \in O_{i,j}^x & \Longleftrightarrow\mbox{\ensuremath{m_{i}} is to the left of \ensuremath{m_{j}}},\label{eq:vi:above:vj}\\
z \in O_{j,i}^x & \Longleftrightarrow\mbox{\ensuremath{m_{i}} is to the right of \ensuremath{m_{j}}},\\
z \in O_{i,j}^y & \Longleftrightarrow\mbox{\ensuremath{m_{i}} is below \ensuremath{m_{j}}},\\
z \in O_{j,i}^y & \Longleftrightarrow\mbox{\ensuremath{m_{i}} is above \ensuremath{m_{j}}},
\end{split}
\end{align}
where
\begin{align}
\begin{split}
     & O_{i,j}^{x}(z):=\{z\in\mathbb{R}^{2N}|\,x_{i}+w_{i}\leq x_{j}\},\\
    & O_{i,j}^{y}(z):=\{z\in\mathbb{R}^{2N}|\,y_{i}+h_{i}\leq y_{j}\}.\label{eq:O_ij^x}
\end{split}
\end{align}
Then this condition is equivalent to the following constraint 
\begin{align}
z\in O_{i,j} & :=O_{i,j}^{x}\cup O_{j,i}^{x}\cup O_{i,j}^{y}\cup O_{j,i}^{y},\label{non-overlapping-constraints}
\end{align}
for $1\le i,j\le N_{m}$ and $i\neq j$.

\noindent\textbf{Condition 3: I/O Assignment.} The coordinates of the I/O pins $p\in\mathcal{S}_{io}$ at $(x_{p}^{pin},y_{p}^{pin})$ are to be determined when considering I/O assignment. 
If the I/O pin $p_{i}$ is at the left boundary of floorplanning region,
then 
\begin{align}
D_{p_{i}}^\mathsf{L}(z):=\left\{ z\in\mathbb{R}^{2N}\ |\ x_{p_{i}}^{pin}=0\;\textrm{and}\;0\leq y_{p_{i}}^{pin}\leq H\right\} .
\end{align}
Similar constraints $D_{p_{i}}^\mathsf{R}(z)$, $D_{p_{i}}^\mathsf{B}(z)$, and $D_{p_{i}}^\mathsf{A}(z)$ can be constructed for I/O pins at the right, bottom
and top boundaries of the floorplanning region. 

Based on the above conditions, we further establish the FSP formulations for the floorplanning (with or without I/O assignment). Combining \textbf{Condition 1} and \textbf{Condition 2}, we define $C_{i,j}$ as:
\begin{align}
\begin{split}
    C_{i,j} & := O_{i,j}\cap B_{i,j}\\
 & \ =\left(O_{i,j}^{x}\cap B_{i,j}\right)\cup\left(O_{j,i}^{x}\cap B_{i,j}\right)\\
 & \  \cup \left(O_{i,j}^{y}\cap B_{i,j}\right)\cup\left(O_{j,i}^{y}\cap B_{i,j}\right)\\
 & :=  C_{i,j,\mathsf{L}}\cup C_{i,j,\mathsf{R}}\cup C_{i,j,\mathsf{B}}\cup C_{i,j,\mathsf{A}},\label{eq:union-convex-set}
\end{split}
\end{align}
where $\mathsf{L},\mathsf{R},\mathsf{B}\:\text{and}\:\mathsf{A}$
stand for the relative relationship of the two modules, i.e., $\mathsf{L}$eft, $\mathsf{R}$ight, $\mathsf{B}$elow, and $\mathsf{A}$bove, respectively.
And the FSP model of the \textbf{floorplanning} becomes:
\begin{align}
\mathrm{Find}\text{ }z\in\bigcap_{1\leq i<j\leq N_{m}}C_{i,j}.\label{eq:feas-seek-basic-floorplanning}
\end{align}
The FSP model of the \textbf{floorplanning with I/O assignment} is further combining \textbf{Condition 3} as:
\begin{equation}
\mathrm{Find}\text{ }z\in(\bigcap_{1\leq i<j\leq N_{m}}C_{i,j})\cap(\bigcap_{p\in\mathcal{S}_{io}}D_{p}),\label{eq:feas-seek-floorplanning-IO-assignment}
\end{equation}
where $D_{p}$ omit the mark in the upper right corner for clarity, and it could either be $D_{p_{i}}^\mathsf{L}(z), D_{p_{i}}^\mathsf{R}(z)$, $D_{p_{i}}^\mathsf{B}(z)$, or $D_{p_{i}}^\mathsf{A}(z)$.

For both feasibility-seeking formulations, we commonly add objective functions intended to decrease the total wirelength of nets in terms of the half-perimeter wirelength (HPWL), which is calculated as follows:
\begin{equation}
  \begin{split}
  \mbox{\textrm{HPWL}}& (x,y):=\\
  \sum\limits_{e\in\mathcal{E}}&(\max\limits_{p,q\in e}\vert x_{p}^{pin}-x_{q}^{pin}\vert + \max\limits_{p,q\in e}\vert y_{p}^{pin}-y_{q}^{pin}\vert).
  \end{split}
\end{equation}

It is worth mentioning that each $C_{i,j}$ in \eqref{eq:feas-seek-basic-floorplanning} and \eqref{eq:feas-seek-floorplanning-IO-assignment} is a union of convex sets. The classical MAP-based algorithm, originally validated for convex FSPs, may in these two problems be unstable. Therefore, we analyze the property and challenge of applying MAP to floorplanning in Sec.~\ref{sec:MAP-property-and-challenges} and propose solutions in Sec.~\ref{subsubsec:rmap}. Besides, to find a superior feasible solution in terms of HPWL for the given FSP-based algorithm, we introduce the enhanced method by the SM in Sec.~\ref{subsubsec:algo-per-rmap}.

\section{Features of Applying MAP to Floorplanning}\label{sec:MAP-property-and-challenges}

In the previous section, we introduced the FSP-based formulations for the common conditions in the floorplanning problem. Considering the complex constraints of the union of convex sets, we describe the property and challenge of applying MAP to the floorplanning problem. Firstly, we analyze the local convergence of MAP for FSP with the unions of convex sets. Then, we 
provide an example to demonstrate the oscillations during the iterations which are the main challenge of applying MAP directly to the floorplanning problem.

\subsection{Local Convergence of MAP}
We found experimentally that in the FSP for unions of convex sets, the initial behavior of the sequential MAP is influenced by the choice of the initialization point of the algorithm. Theorem~\ref{thm:attractor} shows that when the initial point is close enough to a common fixed point of all $P_{i,j}$, i.e., close enough to a feasible point of the FSP, then the sequence will converge to a feasible solution, regardless of the scanning order and the relaxation parameters.

\begin{definition}[Fixed Point Set]
The fixed point set of an operator $\mathcal{T}$ is defined as 
\begin{align}
    F(\mathcal{T}) := \left\{z \in \mathbb{R}^{2N}\ |\ z \in \mathcal{T}(z)\right\}.
\end{align}
\end{definition}

This concept is also associated with the theory of discrete dynamical systems, where fixed points are steady states. Any sequence such that $z^{n+1} \in \mathcal{T}(z^{n})$ is called an orbit or a trajectory.

\begin{definition}[Lyapunov stability]
A steady state $z^*$ is stable in the sense of Lyapunov\cite{stable}, if for every $\epsilon > 0$ there exists a $\delta > 0$ such that every trajectory $z^{n+1} \in T(z^{n})$ starting at $z^{0}\in B(z^*,\delta)$ satisfies $z^{n}\in B(z^*,\epsilon)$ for all $n \ge 1$.
\end{definition}

Here and throughout $B(z,r) = \left\{z'\ |\ \Vert z' - z \Vert_2 < r \right\}$ means the open Euclidean ball with center $z$ and radius $r>0$.

\begin{lemma}
Finding feasible solutions of all constraints in~\eqref{eq:fsp} is to find the common fixed points of all relaxed projections in~\eqref{eq:relaxed-projection}\footnote{We omit $\lambda$ in $\mathcal{T}_{t}(z;\lambda)$ for conciseness in this section.}, i.e.,
\begin{align}
    \text{find } z \in \cap_{t=1}^{M} C_t \Longleftrightarrow \text{find } z \in \cap_{t=1}^{M} F(\mathcal{T}_{t}).
\end{align}
\end{lemma}

\begin{proof}
For any $z\in F(\mathcal{T}_{t})$, the relaxed projection $z \in \mathcal{T}_{t}(z)$ holds, which implies that the projection $z \in \mathcal{P}_{t}(z)$ also holds according to~\eqref{eq:relaxed-projection}, and thus, $z \in C_t$.

Vice versa, for any $z \in C_t$, it is straightforward to verify that $z \in \mathcal{P}_t(z)$, $z \in \mathcal{T}_{t}(z)$, and thus, $z \in F(\mathcal{T}_{t})$.
\end{proof}

\begin{definition}[Active Indices]\label{def:active-indice}
For constraint $C_t = C_{t,\mathsf{L}} \cup C_{t,\mathsf{R}} \cup C_{t,\mathsf{B}} \cup C_{t,\mathsf{A}}$ in~\eqref{eq:union-convex-set}\footnote{The $t$-th constraint $C_t$ is from the non-overlap condition between some pair of modules $m_i$ and $m_j$, i.e., $C_t \equiv C_{i,j}$. We use $t$ and its corresponding $(i,j)$ interchangeably in this section.}, the active indices of $C_t$ at $z$ are
\begin{align}
K_t(z):=\left\{ k\in \left\{\mathsf{L},\mathsf{R},\mathsf{B},\mathsf{A}\right\} \right. \left.  |\ P_{t,k}(z) \in \mathcal{P}_{t}(z) \right\}.
\end{align}
where $P_{t,k}(z) = \mathop{\arg\min}_{c}\{\Vert c - z \Vert\ |\ c \in C_{t,k}\}$.
\end{definition}

Fig.~\ref{fig:illustration-R1-R2} illustrates the notations in Definition~\ref{def:esc-dist} and~\ref{def:cap-dist}, which will be used in both Lemma~\ref{lemma:active} and Lemma~\ref{lemma:active2}. 

\begin{figure}[t]
\centering
\includegraphics[scale=0.45]{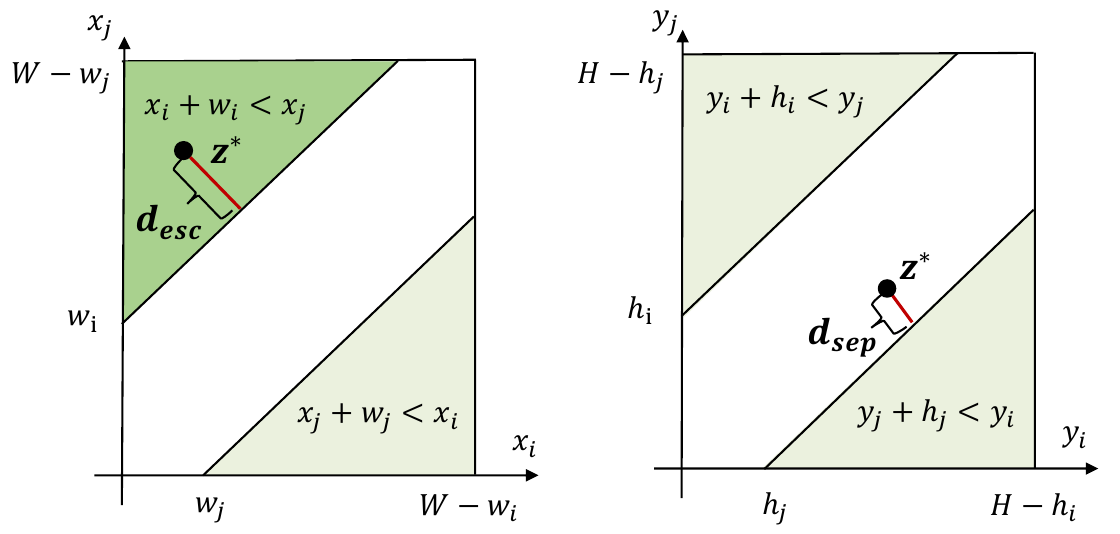}
\caption{Illustration of $d_{esc}(z^*,C_t)$ and $d_{sep}(z^*,C_t)$ when $t=\left(i,j\right)$, where the black point is the point $z^*$. Here, $K_t(z^*)=\left\{\mathsf{L}\right\}$, corresponding to the indice of the dark green region.}%
\label{fig:illustration-R1-R2}%
\end{figure}

\begin{definition}[Escaping Distance from Active Indices]\label{def:esc-dist}
The escaping distance of $z^*$ from all $C_{t,k}$ with an active index $k$ is $d_{esc}(z^*,C_t):=\inf\{\Vert z-z^* \Vert\ |\ \text{$z \notin C_{t,k}$ for $k \in K_{t}(z^*)$}$\}.
\end{definition}
 
\begin{definition}[Separating Distance to Non-active Indices]\label{def:cap-dist}
The separating distance of $z^*$ to all $C_{t,k}$ with a non-active index $k$ is $d_{sep}(z^*,C_t):=\min\{d(z^*,C_{t,k})\ |\ k\not\in K_t(z^*)\}$.
\end{definition}

\begin{lemma}\label{lemma:active}
Given $C_t$ and $z^*\in F(\mathcal{T}_t)$, every $z \in B(z^*,r_t)$ has $K_t(z) = K_t(z^*)$, where \begin{align}
r_t = \min\{d_{sep}(z^*,C_t),\ d_{esc}(z^*,C_t)\}.
\end{align}
\end{lemma}

\begin{proof}
For any $z \in B(z^*,r_t)$, every $k \in K_t(z^*)$ is still an active index at $z$, because $\Vert z - z^* \Vert < d_{esc}(z^*,C_t)$ and  $d(z,C_{t,k})=0$; meanwhile, every $k' \notin K_t(z^*)$ is still non-active at $z$, because $\Vert z - z^* \Vert < d_{sep}(z^*,C_t)$ and $d(z,C_{t,k'}) > 0 = d(z,C_{t,k})$. Therefore, we have $K(z) = K_t(z^*)$.
\end{proof}

Lemma~\ref{lemma:active} gives a rough bound for $r_t$ to support the proof. Lemma~\ref{lemma:active2} in Appendix~\ref{sec:more-local-convergence} presents an improved $r_t$.

\begin{definition}[Sequential MAP Trajectory]
A sequential MAP trajectory is $z^0, z^1, z^2, \ldots$, such that $z^{n+1} \in \mathcal{T}_{t_n}(z^n)$ for $n \ge 0$ and $t_n \in \{1,2,\ldots,M\}$. In addition, sequence $\{\mathcal{T}_{t_n}\}$ is obtained from the MAP method, referring to the relaxed projection in~\eqref{eq:relaxed-projection} onto the union of convex sets $C_{t_n}$ in~\eqref{eq:union-convex-set}.
\end{definition}

Finally, the next Theorem~\ref{thm:attractor} reveals a quantitative property on the influence of the location of the initial point on the behavior of the algorithm.

\begin{theorem} \label{thm:attractor}{\bf(Stable Local Attractor)}
Let $z^* \in \cap_{t=1}^MF(\mathcal{T}_{t})$ be a common fixed point whose radius of attraction is $r > 0$. If for arbitrary fixed $\epsilon \in (0,r)$, a sequential MAP trajectory $z^{n+1} \in \mathcal{T}_{t_{n}}(z^{n})$ enters the ball $B(z^*,\epsilon)$, then the succeeding points in the trajectory stay in the ball and converge to some common fixed point $\tilde{z}^* \in \cap_{t=1}^M F(\mathcal{T}_t)$. The radius of attraction $r$ is at least
\begin{align}
r = \min \left\{r_t\ |\ t\in\{1,2,\ldots,M\}\right\},
\end{align}
where $r_t$ is defined in Lemma~\ref{lemma:active} (or Lemma~\ref{lemma:active2}).
\end{theorem}

\begin{proof}
According to Lemma~\ref{lemma:active} (or Lemma~\ref{lemma:active2}), for any $z^{n} \in B(z^*,\epsilon)$ with $\epsilon \in (0,r)$, $K_{t_n}(z^n) \subseteq K_{t_n}(z^*)$. For $k_n \in K_{t_n}(z^n)$, let $T_{t_{n},k_{n}}(z) := z + \lambda(P_{t_{n},k_{n}}(z)-z)$, which is a single-valued function due to the convexity of $C_{t_n,k_n}$. Thus, $z^{n+1} = T_{t_n,k_n}(z^n) \in B(z^*,\epsilon)$, because
\begin{align*}
\Vert z^{n+1} -z^* \Vert = \Vert T_{t_n,k_n}(z^n) - T_{t_n,k_n}(z^*) \Vert \leq \Vert z^{n}-z^*\Vert.
\end{align*}
The last inequality results from the non-expansive property~\cite{Escalante2011} of the relaxed projection onto closed convex sets.

In words, as soon as the MAP trajectory $z^{n+1}\in \mathcal{T}_t(z^n)$ enters the ball $B(z^*, \epsilon)$, the succeeding points stay in the ball. We show next that they converge to some fixed point in $B(z^*, \epsilon)$, relying on the fact that when the neighborhood of the fixed point is small enough, the MAP trajectory of an FSP with union of convex sets could degrade into a MAP trajectory of FSP with convex sets.

$T_{t_{n},k_{n}}$ is a paracontracting operator (see Definition 1 in~\cite{Elsner1992}, details omitted) since it is the relaxed projection onto the closed convex set $C_{t_{n},k_{n}}$.
By Theorem 1 in~\cite{Elsner1992}, the MAP trajectory $z^{n+1} = T_{t_n,k_n}(z^n)$ converges, if and only if the paracontracting operators $\{T_{t_{n},k_{n}}\}$ have a common fixed point. It is easy to verify that $z^* \in \cap_{t=1}^MF(\mathcal{T}_{t})$ is also a common fixed points of $\{T_{t_{n},k_{n}}\}$.

Therefore, $\tilde{z}^* = \lim_{n \to \infty}z^n$ converges and by the continuity of the distance function, $\tilde{z}^* \in B(z^*, \epsilon)$. Moreover, the theorem also says that the limit $\tilde{z}^*$ would be a common fixed point of $\{T_{t_{n},k_{n}}\}$.
\end{proof}

\subsection{Oscillation Phenomena in MAP}\label{sec:limitations-of-map}
Even with relaxation parameter $\lambda$ and scanning order variations, the MAP sequence may converge to an infeasible solution or oscillate among some infeasible solutions under some initialization. We roughly refer to the two cases as oscillation phenomena. Below are some representative examples to show the phenomena.

\begin{eg}{\bf{\texttt{n4}: 4-module case.}}\label{eg:n4}
Place four modules with widths $w=[4,8,6,4]$ and heights $h=[4,4,4,4]$ in a region of size $(W,H)=(8,12)$. Suppose the initial position is as depicted in the left-hand side (LHS) or right-hand side (RHS) cases of Fig.~\ref{fig:4-cell-oscillation}. In that case, it will result in an oscillation between the two cases when applying MAP with relaxation parameter $\lambda=1$ and using the sequential order of scanning modules from the left bottom to the upper right.
\begin{figure}[ht]
\centering
\includegraphics[scale=0.35]{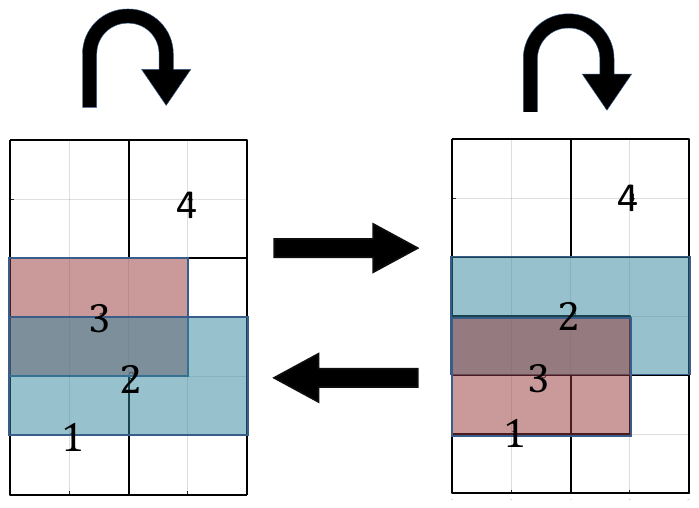}
\caption{Oscillation for \texttt{n4} occurs for bad initialization.}%
\label{fig:4-cell-oscillation}%
\end{figure}

\begin{figure}[ht]
\centering
\includegraphics[scale=0.25]{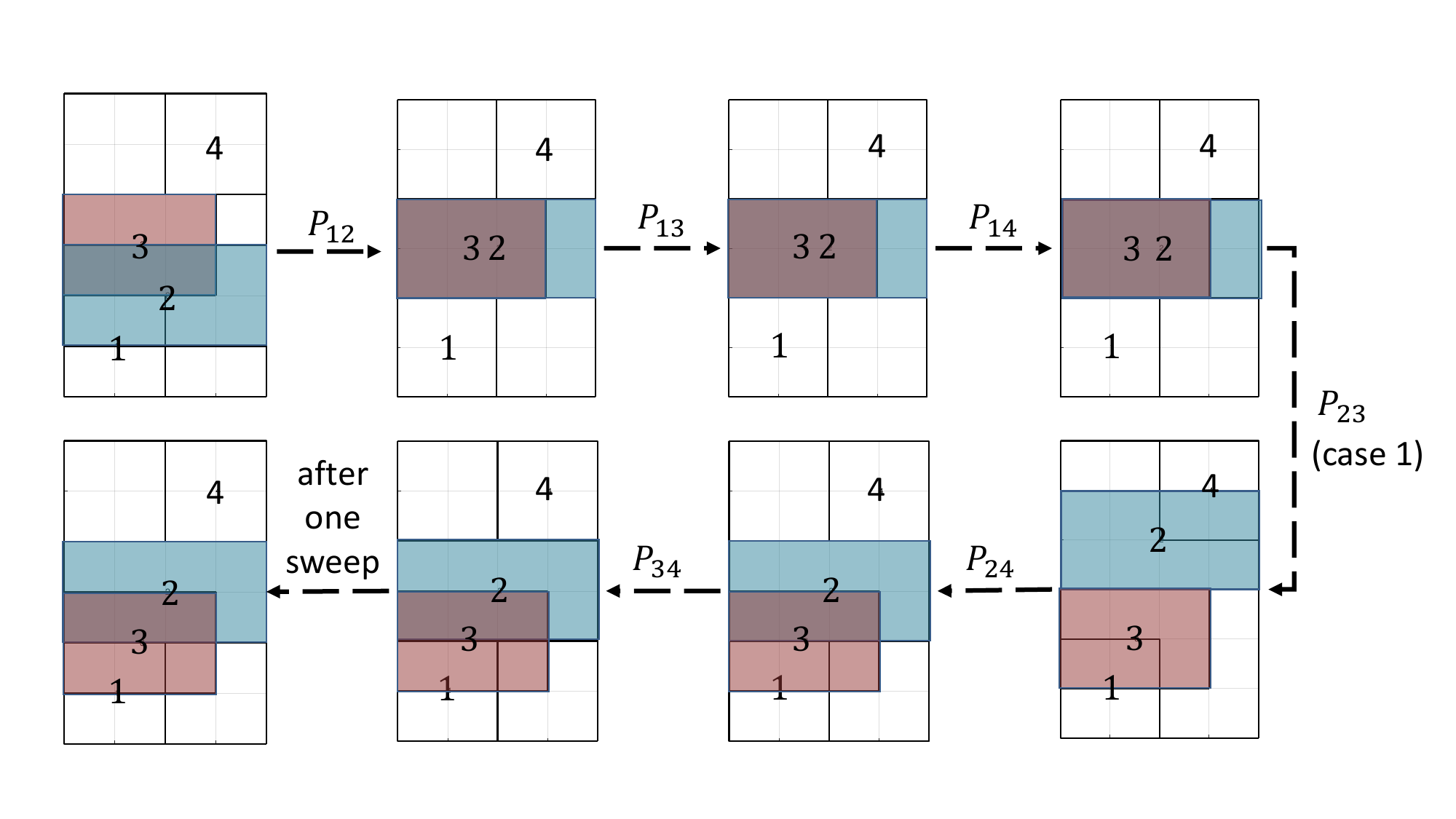} 
\caption{Case 1: begin at the LHS case of Fig.~\ref{fig:4-cell-oscillation}, finish at the RHS case of Fig.~\ref{fig:4-cell-oscillation} after applying a sweep of MAP.}
\label{fig:dead-lock-1}%
\end{figure}

\begin{figure}[ht]
\centering
\includegraphics[scale=0.25]{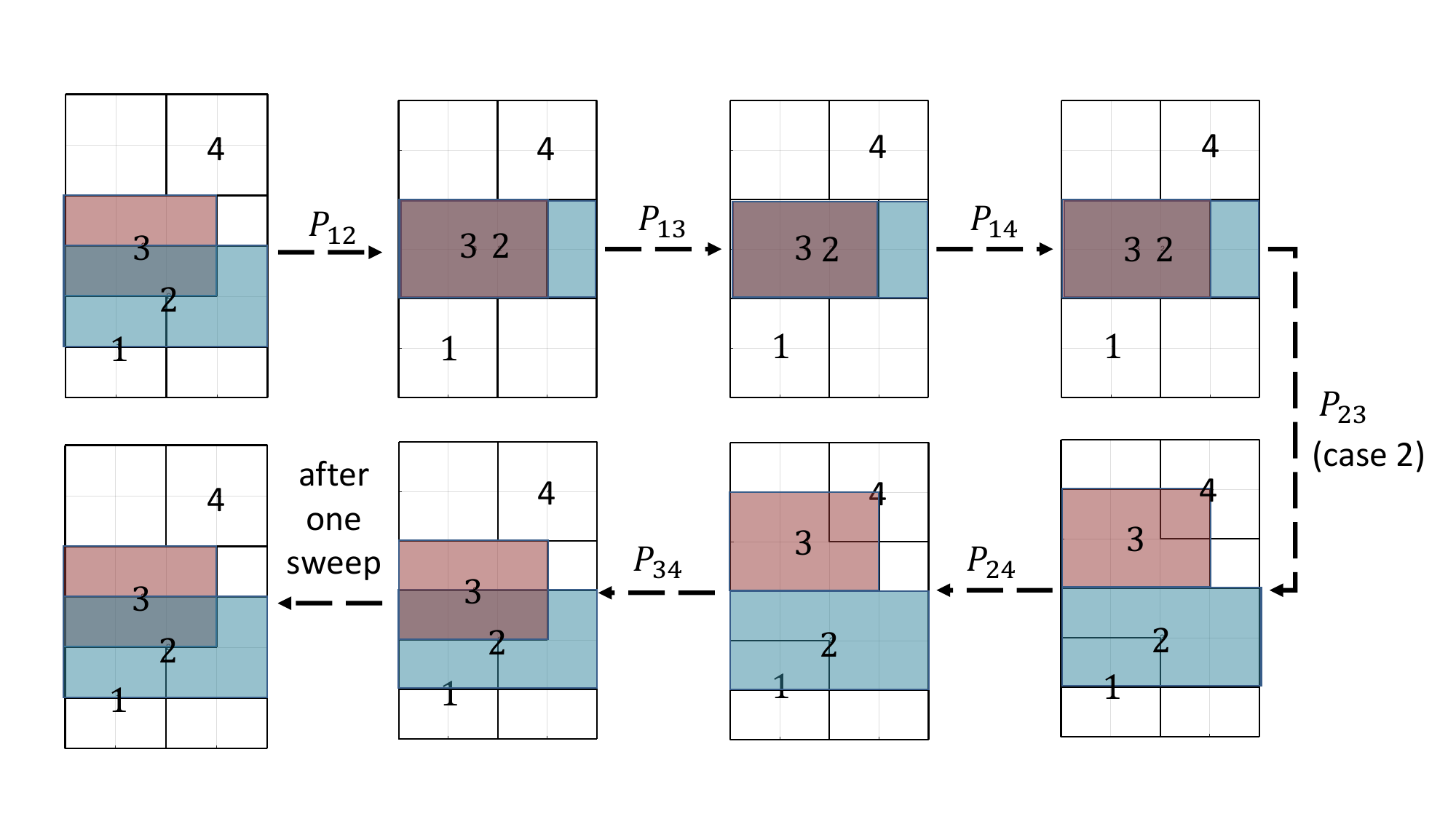}
\caption{Case 2: begin and finish at the LHS case of Fig.~\ref{fig:4-cell-oscillation} after applying a sweep of MAP.}%
\label{fig:dead-lock-2}%
\end{figure}

In one iteration, MAP checks overlaps of a pair of modules and tackles the constraints between them. In one sweep, all constraints would be scanned. When the initial position is the LHS case of Fig.~\ref{fig:4-cell-oscillation}, it may go to the RHS case of Fig.~\ref{fig:4-cell-oscillation} or stay unchanged after a sweep, as shown in Fig.~\ref{fig:dead-lock-1} and~\ref{fig:dead-lock-2}. A similar situation will happen if the initial position is the RHS case of Fig.~\ref{fig:4-cell-oscillation}. As a result, an oscillation occurs.

The oscillations arise regardless of the value of $\lambda$ and regardless of the scanning order of the constraints sets. This is because MAP, employing orthogonal projections, does not alter the relative positions of modules 1-2-4 and 1-3-4.
Consequently, module 1 and module 4 always remain separated, making it impossible to find any feasible solutions for the whole problem.
\end{eg}

In summary, the reasons that we observed for oscillations in applying MAP to FSP with the union of convex sets are:
\begin{enumerate}
    \item The closest point strategy employed by MAP selects the closest convex sets to project onto, which does not guarantee the correct choice of convex subsets. Incorrect choice can lead to a lack of feasible solutions if the intersection of the chosen sets is empty.
    \item The projections onto non-convex sets may not be unique, which results in the diversity of positions of the next iteration.
\end{enumerate}

MAP could encounter divergence and we could see an oscillation phenomena, which is a problem hard to solve by directly adjusting the parameters of MAP. In the next section, a new strategy instead of the “closest point strategy” is designed to eliminate the oscillation problem.

\section{A Perturbed Resettable Method of Alternating Projection (Per-RMAP)}\label{sec:PerRMAP}
Based on the analysis of the convergence behavior of MAP in the previous section, here we present the algorithmic flow based on MAP. As shown in Fig.~\ref{fig:algorithm-framwork},
it consists of three phases: initialization (Sec.~\ref{subsec:algo-init}),
global floorplanning (i.e., Per-RMAP in Sec.~\ref{subsec:algo-global-fplan}),
and post-processing (Sec.~\ref{subsec:algo-post-proc}).

In the initialization phase, module positions with minimized HPWL are achieved by solving quadratic programming, followed by a slight modification of certain modules to the boundary. 
In the global floorplanning phase, an SM called Per-RMAP is designed. It generalizes the MAP with a resetting strategy called RMAP to improve the convergence behavior of MAP, and then perturbations are applied to improve the original FSP-based algorithm RMAP.
In the post-processing phase, considering the diminishing effect of perturbations, Per-RMAP is rerun to further improve the result.
\begin{figure}[t]
\centering 
\includegraphics[scale = 0.45]{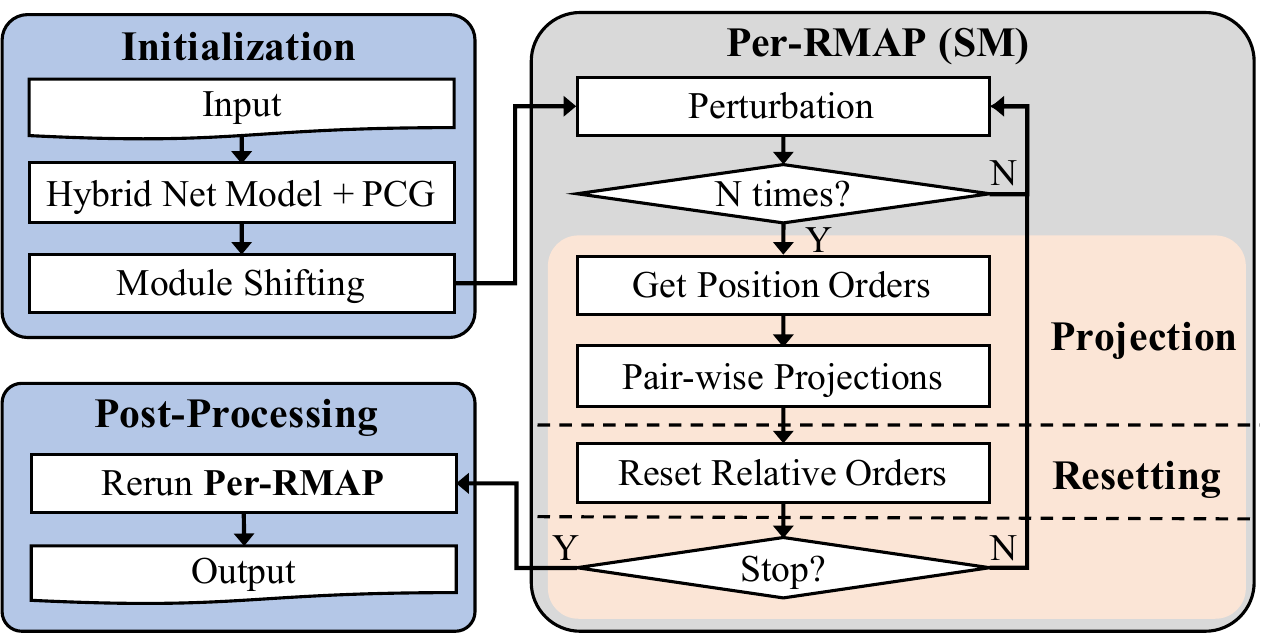}
\caption{Algorithmic flow}
\label{fig:algorithm-framwork}
\end{figure}

\subsection{Initialization}\label{subsec:algo-init}
The initialization of module positions involves two steps: a wirelength minimization step and a shifting step.

In the first step, modules are assigned positions that minimize the HPWL, disregarding module overlaps. We assume initial positions of I/O pins are given\footnote{In our experiment, the initial positions of I/O pins are derived from the benchmarks of the hard macro cases, stored in .pl files. In implementation, the initial positions of I/Os could be generated by chip-package co-design methods, see e.g.~\cite{he2011system}.} and they are fixed during this step. The updated positions of modules are achieved by solving a quadratic programming (QP) problem using the preconditioned conjugate gradients (PCG) method. This approach is similar to other force-directed placers like SimPL~\cite{Kim2012} and FastPlace~\cite{Viswanathan2005}. For net decomposition, the hybrid net model~\cite{Viswanathan2005} is employed. This model combines a classical clique model and a star model, striking a balance between speed and accuracy while capturing the relative positions effectively. However, when there are few connections between modules and I/O pins located on the boundary of the floorplanning region, the modules tend to cluster together.

In the second step, module shifting is utilized to refine the initialization.
The module shifting is only applied to some of the cases with vastly different-sized modules. After the QP procedure, modules tend to cluster at the center of the floorplanning region. If we apply projections to two modules with vastly different sizes, modules of relatively smaller sizes tend to get stuck at the center which may heavily degrade the wirelength. To solve this problem, we move them away from the center. So in implementation, we shift the smallest modules toward the boundary, whose width or height is less than 10\% of the largest one.

\subsection{Global Floorplanning: Per-RMAP}\label{subsec:algo-global-fplan} 
We design our SM algorithm called Per-RMAP that relies on projections. To address challenges in this approach, we design both a resetting strategy and perturbations to achieve an effective and efficient global floorplanning solution.

The resetting strategy mitigates oscillation issues when applying the MAP algorithm to floorplanning with a feasibility-seeking formulation. It identifies pairs of modules that could potentially lead to conflicts.

Additionally, perturbations help find a feasible solution with reduced wirelength. It iteratively improves the objective function while maintaining feasibility.

\subsubsection{Resettable Method of Alternating Projection (RMAP)}\label{subsubsec:rmap}
To address challenges in solving floorplanning problems with non-convex constraint sets using the MAP algorithm in~\eqref{eq:relaxed-projection}, we propose a generalized version and a resetting strategy to overcome issues like converging to an infeasible solution or oscillating among infeasible solutions, which we abbreviate as oscillation phenomena.

To overcome this, we introduce a customizable choice strategy by incorporating a ``preference ratio'' among the subsets within the union. This allows for more control over the selection process. As shown in Algorithm~\ref{alg:RMAP}, the generalized MAP follows a specific order of module scanning and applies projections by taking the weighted average of projections onto four convex sets defined in~\eqref{eq:union-convex-set}, where the preference ratio plays a crucial role in selecting among these convex sets and is amplified using an exponential function.

\begin{algorithm}[t] \small
\caption{
RMAP algorithm
}
\label{alg:RMAP} 
\begin{algorithmic}[1]
\Require
  \Statex The initial positions: $z=(x,y)\in\mathbb{R}^{2N}$
  \Statex  The processing order of constraints: $order$
\Ensure
  \Statex The updated positions: $z=(x,y)$
  \Statex 
\For{$C_{i,j}$ in $order$}
  \State $(\eta_{L},\eta_{R},\eta_{B},\eta_{A}) = preference\_ratio(i,j)$
  \State $w_{t} = {\exp({\eta_{t}}/{\epsilon})}/ {\sum_{k}{\exp({\eta_{k}}/{\epsilon})}}$, for $t\in\{\mathsf{L},\mathsf{R},\mathsf{B},\mathsf{A}\}$
  \State $z=\sum_{t\in\left\{\mathsf{L},\mathsf{R},\mathsf{B},\mathsf{A}\right\}}w_{t} \cdot P_{C_{i,j,t}}(z)$
\EndFor 
\State \Return $z$
\end{algorithmic}
\end{algorithm}

In the MAP algorithm, the preference ratio is determined using the closest point strategy, which is calculated by the expression $\eta_{t}=-\|z-P_{i,j,t}(z)\|$.

Building upon this, a resetting strategy is devised. The resetting strategy could improve the convergence behavior of MAP. It records the subsets selection for each union convex set during previous iterations and based on which the preference ratio is calculated, as depicted in \eqref{eq:resetting-strategy}:
\begin{equation}
    \eta_{k} = 
       \begin{cases}
        -\infty, & \text{if } c_{i,j,k} > S \text{ (and reset } c_{i,j,k} = 0),\\
        -\|z-P_{i,j,k}(z)\|, & \text{o.w. (and } c_{i,j,k} = c_{i,j,k}+1).
        \label{eq:resetting-strategy}
        \end{cases}
\end{equation}
where $S$ is a predefined positive integer. The $c_{i,j,k}$ is the count
of each convex set $C_{i,j,k},\textrm{ for }k\in\{\mathsf{L},\mathsf{R},\mathsf{B},\mathsf{A}\}$ that has been projected onto since the last reset.

That is, when a pair of modules repeatedly projects into specific subsets without successfully removing overlap for more than $S$ times, a ``reset'' action is triggered. This reset action assigns the lowest preference ratio to that direction in the current iteration, aiming to break free from oscillations phenomena.

\subsubsection{Perturbations in Superiorization Method}\label{subsubsec:superiorization}  
The SM lies conceptually between feasibility-seeking and
constrained optimization. While seeking compatibility with constraints,
SM iteratively reduces the value of an objective function without necessarily reaching a minimum. It employs proactive measures to perturb the iterates, guiding them towards a feasible point and simultaneously reducing the objective function.

Drawing inspiration from a modern version of superiorization proposed by Censor et al.~\cite{yair2020}, we present novel modifications of the SM to reduce the HPWL, where the perturbations at $k$-th iteration are shown in Algorithm~\ref{alg:superiorization}.

\begin{algorithm}[t] \small
\caption{
Perturbations in the SM
}
\label{alg:superiorization} 
\begin{algorithmic}[1] \Require \Statex
The intermediate positions: $z=(x,y)\in\mathbb{R}^{2N}$ 
 \Statex Number of perturbations in one iteration: $\mbox{\textit{Num}}$ 
 \Statex Current iteration number: $k$ 
 \Statex Perturbation decay index: $\ell_{k-1}$ 
 \Statex Minimum perturbation length: $\lambda_{min}$ 
 \Statex Initial perturbation length: $\lambda_{init}$
 \Statex Perturbation decay factor: $\Lambda\in(0,1)$ 
 \Ensure \Statex The updated
positions: $z$ \Statex The new perturbation decay index: $\ell_{k}$
\Statex
\For{$n=1:\mbox{\textit{Num}}$}\label{eq:pert-loop-1} 
    \If{$k<l_{k-1}$}\label{eq:pert-decay-index-begin}
    \State $\ell_{k} = \text{a random integer in $\left[k,\ell_{k-1}\right]$}$
    \Else \State $\ell_{k}=k$ 
    \EndIf\label{eq:pert-decay-index-end}
 \State $v^{k,n}=\nabla \mbox{\textrm{HPWL}}(x,y)$ 
 \For{$cnt=1:10$} 
 \State $\lambda_{pert}^{k,n}=max(\lambda_{min},\lambda_{init}\cdot\Lambda^{l_{k}})$\label{eq:lambda-pert} 
\State $(x',y')=z'=z-\lambda_{pert}^{k,n}\cdot{v^{k,n}}/{\|v^{k,n}\|}$\label{eq:pos-pert}
\If{$\mbox{\textrm{HPWL}}(x',y')<\mbox{\textrm{HPWL}}(x,y)$} 
\State $z=z'$ 
\State \textbf{break} // and continue at line \ref{eq:pert-decay-index-begin}
\EndIf
\State $\ell_{k}=\ell_{k}+1$ 
\EndFor \EndFor \State \Return
$(z,\ell_{k})$ 
 \label{eq:pert-loop-2} 
\end{algorithmic} 
\end{algorithm}

At iteration $k$, the SM algorithm applies $\mbox{\textit{Num}}$ perturbations to adjust the positions of modules and I/O pins in accordance with the negative gradient of the HPWL function. If a perturbation (line \ref{eq:lambda-pert}-\ref{eq:pos-pert}) reduces the total wirelength, it is accepted. However, if the wirelength does not decrease, the algorithm iteratively decreases the step size up to a maximum of 10 times until a shorter wirelength is achieved.

The step sizes $\lambda_{pert}^{k,n}$ (line
\ref{eq:lambda-pert}) decreases with iterations and has the following properties:
\begin{enumerate}
\item They should be summable,
i.e., $\sum_{k=0}^{\infty}\sum_{n=0}^{\footnotesize{\mbox{\textit{Num}}}-1}\lambda_{pert}^{k,n}<+\infty$. Here we generated via a subsequence of
 $\{\Lambda^{\ell}\}_{\ell=0}^{\infty}$ with $\ell$ powers of the user-chosen kernel $0<\Lambda<1$. 
\item A lower bound $\lambda_{min}$ is given to ensure the performance
of the perturbation. In our setting, $\lambda_{min}=0.1$. 
\item The perturbation decay index $\ell_k$ controls the step size of the perturbations, as shown in line~\ref{eq:lambda-pert} of Algorithm~\ref{alg:superiorization}. Since $\ell$ increases after each inner loop, the perturbation magnitude could converge to zero quickly. To preserve the impact of perturbations in later iterations, the exponent $\ell_k$ is adjusted to a random integer between current iteration number $k$ and previous value $\ell_{k-1}$, as shown in lines~\ref{eq:pert-decay-index-begin}-\ref{eq:pert-decay-index-end} of Algorithm~\ref{alg:superiorization}. This modification was also used in linear superiorization (LinSup)\cite{censor2017}  and total variation superiorization applied to reconstruction in proton computed tomography\cite{yair2020}.
\end{enumerate}

\subsubsection{Perturbed Resettable Method of Alternating Projection (Per-RMAP)}\label{subsubsec:algo-per-rmap} 
Combining RMAP of Algorithm~\ref{alg:RMAP} as the feasibility-seeking component and the perturbations proposed in Algorithm~\ref{alg:superiorization}, we obtain our SM algorithm, called Perturbed Resettable Method of Alternating Projection (Per-RMAP), as presented in Algorithm~\ref{alg:PerRMAP}.
It offers a comprehensive methodology for reaching improved positions of modules and I/O pins in a floorplan. 
\begin{algorithm}[t] \small
\caption{
Per-RMAP algorithm
}
\label{alg:PerRMAP} \begin{algorithmic}[1] \Require 
\Statex The initial positions: $z=(x,y)\in\mathbb{R}^{2N}$ 
\Statex Number of perturbations in one iteration: $\mbox{\textit{Num}}$ 
\Statex Minimum perturbation length: $\lambda_{min}$ 
\Statex Initial
perturbation length: $\lambda_{init}$ 
\Statex Perturbation decay factor: $\Lambda\in(0,1)$ 
\Statex Initial projection length: $\gamma_{init} \in (0,1)$
\Statex Projection progress factor: $\Gamma>1$
\Ensure \Statex The updated positions $z=(x,y)$ \Statex
\State $\ell_{-1}=0$
\For{$k=0:\infty$}
\State $(z,\ell_{k})=\mbox{\textrm{Perturbation}}(z,\mbox{\textit{Num}},k,\ell_{k-1},\lambda_{min},\lambda_{init},\Lambda)$
\State $order=$ {[}generated by position order{]}
\State $\gamma_{proj}=\min(1,\gamma_{init}\cdot\Gamma^{k})$\label{eq:lambda-proj}
 \State $z=z+\gamma_{proj}\cdot(\mbox{\textrm{RMAP}}(z,order)-z)$ 
 \State $isStop = \mbox{\textrm{RelativeOverlappingAreaCheck}}(z)$
 \If{$isStop == true$}
 \State \textbf{break}
 \EndIf
\EndFor
\State \Return $z$
\end{algorithmic} 
\end{algorithm}

Per-RMAP focuses on feasibility-seeking, and improves wirelength by perturbations interlaced into the feasibility-seeking procedure. This method is composed of lightweight projections. Each individual set allows for easy and fast projections onto it. As a result, compared to other conventional methods which are designed to seek optimality and are usually composed of intermediate steps of objectives processing, Per-RMAP significantly enhances the algorithm efficiency. Persistent perturbations make the solution quality competitive with other methods and with the time cost savings.

The termination condition of Per-RMAP is based on the behavior of the ``relative overlapping area'' (ROA). If the ROA is below a threshold, indicating a feasible solution, the algorithm terminates and returns the current positions. If the ROA remains constant for a period of consecutive iterations or oscillates among certain values, which indicts of oscillations, the algorithm terminates and returns error.

\subsection{Post Processing}\label{subsec:algo-post-proc}
After global floorplanning, we obtain a feasible result. However, the application of the superiorization technique, which gradually brings modules closer together, has a diminishing effect on position changes as iterations progress. Consequently, there may still be small gaps remaining between the modules.

Hence Per-RMAP is rerun to improve the result. At this phase, the perturbation decay index is reset to $k\times\theta$, where $k$ is the total iteration number and
$\theta\in(0,1)$. This reset allows for a fresh start and helps to close any remaining gaps between the modules.

For the final step of floorplanning with the I/O assignment, we legalize the I/O pins to predefined positions that are derived from a pin pitch. The order of I/O pins is kept the same as that in global floorplanning.

\section{Experimental Evaluations}\label{sec:experiment}
In this section, we introduce the comprehensive evaluations for Per-RMAP and the other key prior studies. 

\subsection{Setup}
\textbf{Benchmarks:} The benchmarks used for evaluations include: (a) Microelectronics Center of North Carolina (MCNC) \cite{mcnc} benchmark; (b) Gigascale Systems Research Center (GSRC)\cite{gsrc} benchmark; and (c) Synthetic benchmark with four cases. The MCNC and GSRC are widely used benchmarks for floorplanning, while the synthetic one is used to illustrate the impact of the resetting strategy. The details of the benchmarks are listed in Table~\ref{tab:bench}, including the number of modules, I/O pins, pins, nets, and the size of the floorplanning region (die size).

\textbf{Baselines:} We consider the following representative prior studies as the baselines according to the classifications mentioned in Sec.~\ref{subsec:related-work}: (a) Branch-and-bound method (B\&B)~\cite{funke2016}, a discrete exact method; (b) SA~\cite{Liu2008}, Fast-SA~\cite{chen2006}, and Parquet-4~\cite{Adya2003}, three discrete meta-heuristic methods with the simulated annealing algorithm; (c) FD~\cite{Samaranayake2009}, an analytical method with force-directed techniques; (d) PeF~\cite{Li2022}, a non-linear analytical method. These baselines combined with the previous version of Per-RMAP (marked as Per-RMAP~\cite{yu2023}) are used for the evaluations of HPWL and runtime to demonstrate the superiority and efficiency of Per-RMAP. Besides, we also conduct ablation studies to evaluate the effectiveness of the resetting strategy and perturbations in Per-RMAP. 

\textbf{Implementations:} The implementations of Per-RMAP in this article are developed in C++ and all of the experiments are evaluated on a server equipped with 2-way Intel Xeon Gold 6248R@3.0GHz CPUs and 768GB DDR4-2666MHz memory. Compared to the previous version of Per-RMAP~\cite{yu2023} (implemented in MATLAB), we refactor the code in C++ and use efficient libraries to speed it up. Besides, the algorithm is modified to enable a simultaneous reset.

\begin{table}[t] \scriptsize
\renewcommand{\arraystretch}{1.08}
\setlength\tabcolsep{3pt}
\centering
\caption{Characteristics of the benchmarks for evaluation.}
\label{tab:bench} %
\begin{tabular}{ccccccc}

\hline\hline
\textbf{Benchmark} & \textbf{Instance}  & \textbf{\#Modules}  & \textbf{\#I/O Pins} & \textbf{\#Pins}  & \textbf{\#Nets} & \textbf{Die Size ($\mu m^{2}$)}\\
\hline
\multirow{5}{*}{MCNC} & \texttt{apte}  & 9  & 73  & 214  & 97 & 10500$\times$10500\\
& \texttt{xerox}  & 10  & 2  & 696  & 203 & 5831$\times$6412\\
& \texttt{hp}  & 11  & 45  & 264  & 83 & 4928$\times$4200\\
& \texttt{ami33}  & 33  & 42  & 480  & 123 & 2058$\times$1463\\
& \texttt{ami49}  & 49  & 22  & 931  & 408 & 7672$\times$7840\\
\hline
\multirow{3}{*}{GSRC}& \texttt{n100}  & 100  & 334  & 1873  & 885 & 800$\times$800\\
& \texttt{n200}  & 200  & 564  & 3599  & 1585 & 800$\times$800\\
& \texttt{n300}  & 300  & 569  & 4358  & 1893 & 800$\times$800\\
\hline
\multirow{4}{*}{Synthetic}&\texttt{n3}  & 3  & -  & -  & - & 11$\times$11\\
&\texttt{n3v}  & 3  & -  & -  & - & 5$\times$11\\
&\texttt{n4}  & 4  & -  & -  & - & 8$\times$12\\
&\texttt{n5}  & 5  & -  & -  & - & 3$\times$3\\
\hline\hline
\end{tabular}
\end{table}

\subsection{Floorplanning Results on the MCNC Benchmark}\label{subsec:result-MCNC}
\subsubsection{Results of Original Floorplanning}
For the evaluations of the fixed-outline floorplanning,
we compared Per-RMAP against SA~\cite{Liu2008}, FD~\cite{Samaranayake2009}, B\&B~\cite{funke2016}, PeF~\cite{Li2022}, and RMAP (Per-RMAP w/o introducing SM perturbations) on the MCNC benchmark.
Per-RMAP is the most efficient method compared with all other methods while achieving sub-optimal results. The HPWL of Per-RMAP achives an average 166$\times$ speed-up with only a 5$\%$ increase compared to B\&B~\cite{funke2016}. In comparison with the state-of-the-art analytical method PeF, Per-RMAP outperforms in both HPWL and time. The floorplanning results of \texttt{ami49} are depicted in Fig.~\ref{fig:ami49-plot}.

\begin{table*}[ht]\scriptsize
\renewcommand{\arraystretch}{1.08}
\caption{Experimental results for floorplanning on the MCNC benchmark.}
\label{tab:floorplan-result-wo-io-1}
\centering \setlength{\tabcolsep}{5pt}
\begin{tabular}{cccccccc} 
\hline\hline
\multirow{2}{*}{\textbf{Instance}} & \multicolumn{7}{c}{\textbf{HPWL}}\\ 
\cline{2-8} 
& SA~\cite{Liu2008}    & FD~\cite{Samaranayake2009}  & B\&B~\cite{funke2016} & PeF~\cite{Li2022} & RMAP & Per-RMAP~\cite{yu2023} & Per-RMAP \\  

\hline
\texttt{apte} & 614602 & 545136 & 513061 & 529162 & 611049 & 528618 & 522331 \\
\texttt{xerox} & 404278 & 755410 & 370993 & 422623 & 932498 & 382596 & 398027 \\
\texttt{hp}& 253366 & 155463 & 153328 & 157204 & 191484 & 159979 & 152926 \\
\texttt{ami33} & 96205  & 63125  & 58627  & 67325 & 122371 & 61444  & 63079 \\
\texttt{ami49} & 1070010 & 871128 & 640509 & 789270 & 1692930 & 637098 & 689296 \\ 
\hline
\textbf{Norm. Ratio}   & \textbf{1.00$\times$} & \textbf{0.97$\times$} & \textbf{0.71$\times$} & \textbf{0.79$\times$} & \textbf{1.38$\times$} & \textbf{0.73$\times$} & \textbf{0.75$\times$} \\
\hline\hline
\multirow{2}{*}{\textbf{Instance}} & \multicolumn{7}{c}{\textbf{时间（秒）}} \\
\cline{2-8}
& SA~\cite{Liu2008} & FD~\cite{Samaranayake2009} & B\&B~\cite{funke2016} &  PeF~\cite{Li2022} & RMAP & Per-RMAP~\cite{yu2023} & Per-RMAP\\
\hline
\texttt{apte} & 1.21 & 0.38& 13.00 & 0.25 & \textless0.01 & 8.84  & 0.03\\
\texttt{xerox} & 0.95 & 0.09 & 48.00 & 0.60 & \textless0.01 & 172.09& 0.43\\
\texttt{hp} & 1.36 & 0.36 & 102.00 & 0.38 & \textless0.01 & 5.06  & 0.11\\
\texttt{ami33} & 1.64& 0.75 & 13.00 & 0.48 & \textless0.01 & 23.83 & 0.14\\
\texttt{ami49} & 5.60& 4.11 & 73.00 & 2.21 & \textless0.01 & 261.76& 1.70\\
\hline
\textbf{Norm. Ratio} & \textbf{1.00$\times$} & \textbf{0.37$\times$} & \textbf{31.45$\times$} & \textbf{0.36$\times$} & - & \textbf{50.69$\times$} & \textbf{0.19$\times$}\\
\hline\hline
\end{tabular}
\end{table*}

\subsubsection{Results of Floorplanning with I/O assignment}
Besides, we further consider the floorplanning problem with the I/O assignment, showing the flexibility of the framework considering extra constraints.
As shown in Table~\ref{tab:floorplan-result-with-io}, 
Per-RMAP with I/O assignment achieves an average 6$\%$ improvement in HPWL and achieves a speed-up of $\sim$100$\times$ compared to B\&B~\cite{funke2016}.
Three of five instances achieve better HPWL. For the \texttt{apte} instance, we achieve at most a 22\% improvement in HPWL. For \texttt{xerox} instance, the HPWL of Per-RMAP is worse than that of B\&B~\cite{funke2016} after considering the I/O assignment. The reason is that there are only two I/O pins on the center of the top and bottom boundaries of the floorplanning region with no room for improvement.

\begin{table}[ht] \scriptsize
\renewcommand{\arraystretch}{1.08}
\caption{Experimental results for floorplanning on the MCNC benchmark \textbf{with I/O assignment}.}
\label{tab:floorplan-result-with-io} \centering
\begingroup
    \setlength{\tabcolsep}{0.3\tabcolsep}
\begin{tabular}{cccccccc}
\hline\hline
\multirow{3}{*}{\textbf{Instance}} & \multicolumn{3}{c}{\textbf{Norm. HPWL}} &  & \multicolumn{3}{c}{\textbf{Time (sec)}} \\
\cline{2-4} \cline{6-8}
& \multicolumn{1}{c}{\multirow{2}{*}{\begin{tabular}[c]{@{}c@{}}B\&B~\cite{funke2016}\end{tabular}}} & \multicolumn{1}{c}{\multirow{2}{*}{\begin{tabular}[c]{@{}c@{}}Per-RMAP \\ w/o I/O\end{tabular}}} & \multicolumn{1}{c}{\multirow{2}{*}{\begin{tabular}[c]{@{}c@{}}Per-RMAP \\ w/ I/O\end{tabular}}} & & \multicolumn{1}{c}{\multirow{2}{*}{\begin{tabular}[c]{@{}c@{}}B\&B~\cite{funke2016}\end{tabular}}}  & \multicolumn{1}{c}{\multirow{2}{*}{\begin{tabular}[c]{@{}c@{}}Per-RMAP \\ w/o I/O\end{tabular}}} & \multicolumn{1}{c}{\multirow{2}{*}{\begin{tabular}[c]{@{}c@{}}Per-RMAP \\ w/ I/O\end{tabular}}} \\
\\ \hline

\texttt{apte} & 1.00$\times$ & 1.02$\times$ & 0.78$\times$ & & 13.00  & 0.03 & 0.08\\
\texttt{xerox} & 1.00$\times$ & 1.07$\times$ & 1.01$\times$ & & 48.00 & 0.43 & 0.59\\
\texttt{hp} & 1.00$\times$ & 1.00$\times$ & 0.93$\times$ & & 102.00 & 0.11 & 0.23\\
\texttt{ami33} & 1.00$\times$ & 1.08$\times$ & 1.01$\times$ & & 13.00 & 0.14 & 0.12\\
\texttt{ami49} & 1.00$\times$ & 1.08$\times$ & 0.97$\times$ & & 73.00 & 1.70 & 1.57\\

\hline
\textbf{Norm. Ratio} & \textbf{1.00$\times$} & \textbf{1.05$\times$} & 
\textbf{0.94$\times$} & &
\textbf{1.00$\times$} & 
\textbf{0.01$\times$} & 
\textbf{0.01$\times$}\\

\hline\hline

\end{tabular}
\endgroup
\end{table}

\begin{figure}[t]
\centering
\includegraphics[scale=0.4]{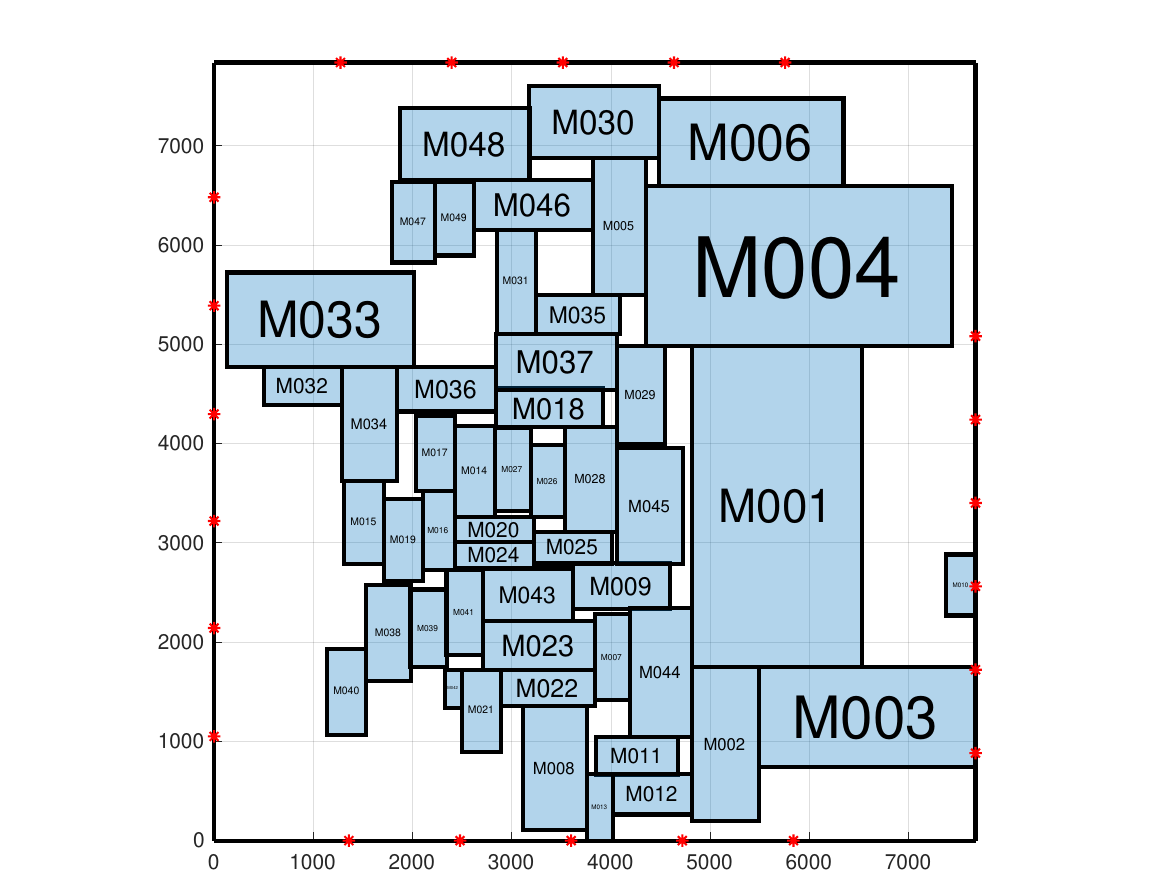} 
\caption{Module packings of \texttt{ami49} for fixed-outline floorplanning without I/O assignment.}
\label{fig:ami49-plot}
\end{figure}

\subsubsection{Results of Floorplanning Considering Soft Modules}
To test the scalability of multiple constraints, we also consider the experiments with soft modules.
Table~\ref{tab:floorplan-soft} presents experimental results for floorplanning with soft modules on MCNC benchmark, which fixes the whitespace at 15\% and the outline aspect ratio as 1:1. Besides, we set the aspect ratio of all soft modules with lower bound $\frac{1}{3}$ and upper bound $3$ as PeF does.

Compared with PeF, Per-RMAP achieves a speed-up of 15\% with comparable HPWL. As an analytical method, although PeF could quickly spread modules under its formulation as the analytical solution is known, it would take some time to keep a trade-off between wirelength objectives and density function. However, our Per-RMAP focuses on feasibility-seeking and improves given objectives during the feasibility-seeking procedure, thus achieving efficiency.

\begin{table}[t] \scriptsize
\renewcommand{\arraystretch}{1.08}
\caption{Experimental results for floorplanning with soft modules, aspect ratio 1:1, whitespace 15\%.}
\label{tab:floorplan-soft} \centering
\begin{tabular}{cccccc}
\hline\hline
\multirow{2}{*}{\textbf{Instance}} & \multicolumn{2}{c}{\textbf{HPWL}}                                      & & \multicolumn{2}{c}{\textbf{Time (sec)}} \\
\cline{2-3} \cline{5-6}
         & PeF~\cite{Li2022} & Per-RMAP & & PeF~\cite{Li2022} & Per-RMAP \\
\hline
\texttt{apte} & 422312 & 390338 & & 1.27 & 0.22\\
\texttt{xerox}& 434925 & 471242 & & 1.07 & 1.35\\
\texttt{hp}   & 150762 & 135778 & & 0.94 & 0.67\\
\texttt{ami33}& 50442 & 46762 & & 0.99 & 0.83\\
\texttt{ami49}& 613665 & 734753 & & 2.49 & 3.18\\
\hline
\textbf{Norm. Ratio} & \textbf{1.00$\times$} & \textbf{1.00$\times$} & & \textbf{1.00$\times$} & \textbf{0.85$\times$} \\
\hline\hline
\end{tabular}
\end{table}

\subsection{Discussion on Effectiveness and Efficiency}
Per-RMAP showcases its effectiveness and efficiency by surpassing some existing approaches in terms of both floorplan quality and execution wall time. Our evaluation is conducted on the GSRC benchmarks for original floorplanning, and the results are shown in Table~\ref{tab:GSRC-floorplan-result-wo-io}. We compare Per-RMAP against Parquet-4~\cite{Adya2003} and Fast-SA~\cite{chen2006}, which are commonly employed in the field and could be implemented in the same settings as ours.

In comparison to Parquet 4.0, our method achieves a 34\% enhancement in HPWL with acceptable time costs. Unlike Parquet 4.0 that uses clustering for efficiency at the cost of quality, our flat approach Per-RMAP excels in quality.

Compared to Fast-SA, our method achieves a 5.6\% runtime of Fast-SA within a 1\% margin in HPWL. Per-RMAP outperforms Fast-SA in finding suboptimal solutions quickly. While Fast-SA focuses on constraint optimization for better quality, it incurs higher time costs. This efficiency of our method stems from our projection-based approach, which is computationally light and focused on feasibility, with slight perturbations to improve solutions.

In summary, our approach provides efficient solutions with competitive HPWL. It is, thus, well-suited for practical floorplanning challenges, where efficiency, solution quality, and the ability to handle larger designs are pivotal considerations. In future work, a block iterative version or string average version could be designed to improve the scalability of the algorithm.

\begin{table*}[ht] \scriptsize
\renewcommand{\arraystretch}{1.08}
\caption{Experimental results for floorplanning on the GSRC benchmark.}
\label{tab:GSRC-floorplan-result-wo-io} \centering
\begin{tabular}{cccccccc}
\hline\hline
\multirow{2}{*}{\textbf{Instance}} & \multicolumn{3}{c}{\textbf{HPWL}}                                       & & \multicolumn{3}{c}{\textbf{Time (sec)}} \\
\cline{2-4} \cline{6-8}
         & Parquet-4~\cite{Adya2003} & Fast-SA\cite{chen2006}  & Per-RMAP & &  Parquet-4~\cite{Adya2003} & Fast-SA\cite{chen2006}  & Per-RMAP \\
\hline
\texttt{n100}  & 420936 & 287646 & 282596 & & 0.33 & 10.72 & 1.01  \\
\texttt{n200}  & 763802 & 516057 & 518722 &  & 1.69 & 69.86 & 2.93\\
\texttt{n300}  & 1010000 &  603811 & 626061 & & 4.81 & 133.63 & 4.11\\
\hline
\textbf{Norm. Ratio} & \textbf{1.00$\times$}  & \textbf{0.65$\times$} &  \textbf{0.66$\times$} & & \textbf{1.00$\times$} & \textbf{33.87$\times$}  & \textbf{1.88$\times$} \\
\hline\hline
\end{tabular}
\end{table*}

\subsection{Ablation Study}\label{subsec:ablation-study}
To evaluate the effectiveness of the proposed Per-RMAP, we conduct the ablation study in this subsection. Here we consider two additional baselines: (a) MAP: the naive MAP algorithm without resetting or perturbations; and (b) RMAP: the MAP algorithm with the proposed resetting strategy.

\textbf{Evaluations of Resetting Strategy:} In our design, the resetting strategy aims to eliminate oscillations and improve the convergence behavior. Therefore, we compare the number of iterations required for convergence and ROA achieved by MAP and RMAP on the synthetic benchmark and the MCNC benchmark. The results are listed in Table~\ref{tab:compare-MAP-and-RMAP}. It is clear that except in the \texttt{hp} case, MAP does not converge in 100 iterations. In contrast, RMAP can find feasible solutions within significantly fewer iterations. This is evident from the reduced number of iterations and the smaller ROA achieved by RMAP compared to MAP. For instance, in the \texttt{n3} and \texttt{n4} instances, the resetting strategy reduces the number of iterations by more than half and achieves a ROA of less than 0.1$\%$, indicating the successful removal of overlaps. 
Similarly, RMAP finds feasible solutions after 33, 35, 19, 50, and 93 iterations on the instances of MCNC benchmarks respectively.

The oscillation can be greatly mitigated due to the resetting strategy. However, it is important to note that the resetting strategy may not eliminate oscillations in some extreme cases such as \texttt{n5} in Appendix~\ref{sec:example-n5}, where the whitespace is zero. In such cases, additional strategies like a multi-start strategy may be necessary to escape the oscillation. The importance of initial point selection on convergence is elaborated in Section~\ref{sec:MAP-property-and-challenges}. Here we do not elaborate on the strategy, as our proposed strategies are effective for most cases.

\begin{table}[t] \scriptsize
\renewcommand{\arraystretch}{1.08}
    \caption{Evaluations of the resetting strategy on the synthetic and MCNC benchmarks.}
    \label{tab:compare-MAP-and-RMAP}
    \centering
    \begin{tabular}{cccccc}
        \hline\hline
        \multirow{2}{*}{\textbf{Instance}} & \multicolumn{2}{c}{\textbf{MAP}} & & \multicolumn{2}{c}{\textbf{RMAP}}\\
        \cline{2-3} \cline{5-6}
         & \textbf{\#Iterations} & \textbf{ROA} & & \textbf{\#Iterations} & \textbf{ROA}\\
        \hline
        \multirow{1}{*}{\texttt{n3}}  & 100 & 1.7\% &  & 31 & \textless \  0.1\%\\
        \multirow{1}{*}{\texttt{n4}}  & 100 & 20.8\% & & 42 & \textless \  0.1\%\\
        \multirow{1}{*}{\texttt{n5}}  & 100 & 11.1\%/22.2\%     
         & & 100 & 11.1\%/22.2\%\\
        \multirow{1}{*}{\texttt{apte}}  & 148 & 4.6\%        
         & & 33 & \textless \  0.1\%\\
        \multirow{1}{*}{\texttt{xerox}}  & 143 & 10.9\%
         & & 35 & \textless \  0.1\%\\
        \multirow{1}{*}{\texttt{hp}} & 38 & \textless \  0.1\%
         & & 19 & \textless \  0.1\%\\
        \multirow{1}{*}{\texttt{ami33}} & 624 & 0.5\%
         & & 50 & \textless \  0.1\%\\
        \multirow{1}{*}{\texttt{ami49}} & 858 & 9.0\%
         & & 93 & \textless \  0.1\%\\
        \hline\hline
    \end{tabular}
\end{table}

\textbf{Evaluations of SM:} 
In the proposed SM algorithm, Per-RMAP, perturbations are applied to guide the FSP-based algorithm to reach a superior point with reduced HPWL. Although the FSP focuses on satisficing rather than optimizing, superiorization makes FSP algorithms competitive with the B\&B method in~\cite{funke2016} which seeks optimality of total wirelength, as shown in Table~\ref{tab:floorplan-result-wo-io-1}. In general, the pure feasibility-seeking algorithm RMAP finds feasible solutions with higher HPWL. Here, with the aid of superiorization, Per-RMAP finds superior feasible solutions. Per-RMAP even halves the HPWL in \texttt{xerox}, \texttt{ami33}, and \texttt{ami49}.
The changes are far less dramatic in \texttt{apte} and \texttt{hp} because the feasible solution found by the RMAP is near optimal and leaves little room for improvement.

\textbf{Comparison of MAP, RMAP, and Per-RMAP on \texttt{ami49}:} 
Fig.~\ref{fig:MAP-RMAP-PerRMAP} illustrates the changes in overlapping area and HPWL over iterations for the \texttt{ami49} benchmark, comparing the performance of MAP, RMAP, and Per-RMAP algorithms.
For MAP, the overlapping area decreases rapidly in the first 100 iterations. However, after that, it starts to oscillate among 14 positions at an average overlapping area of $2.04\times 10^6$ and fails to find a feasible solution.
In contrast, RMAP gradually reduces the overlapping area to zero as iterations progress. However, the HPWL gradually increases to $1.70 \times 10^6$. This suggests that while RMAP successfully finds a feasible solution, it does not effectively reduce the HPWL.

Per-RMAP utilizes the superiorization technique to achieve a remarkable reduction in HPWL, resulting in a value of $6.37\times 10^5$. This improvement comes at the cost of a slower decline in the overlapping area as the algorithm progresses. The projection steps gradually become more dominant, ultimately reaching a feasible solution.
\begin{figure*}[t]
    \centering
    \begin{subfigure}[t]{0.32\textwidth}
        \centering
        \includegraphics[width=\textwidth]{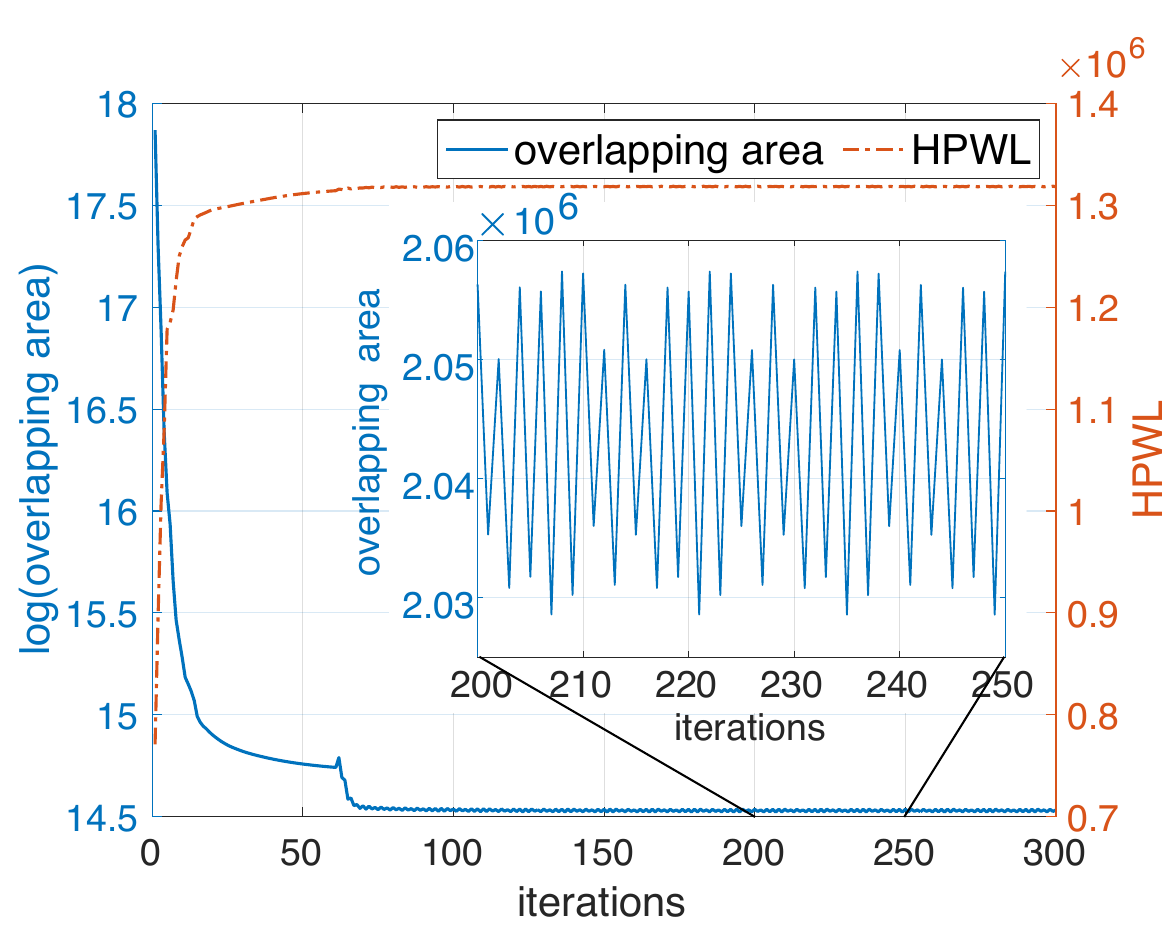}
        \caption{MAP}
    \end{subfigure}
    \hfill
    \begin{subfigure}[t]{0.32\textwidth}
        \centering
        \includegraphics[width=\textwidth]{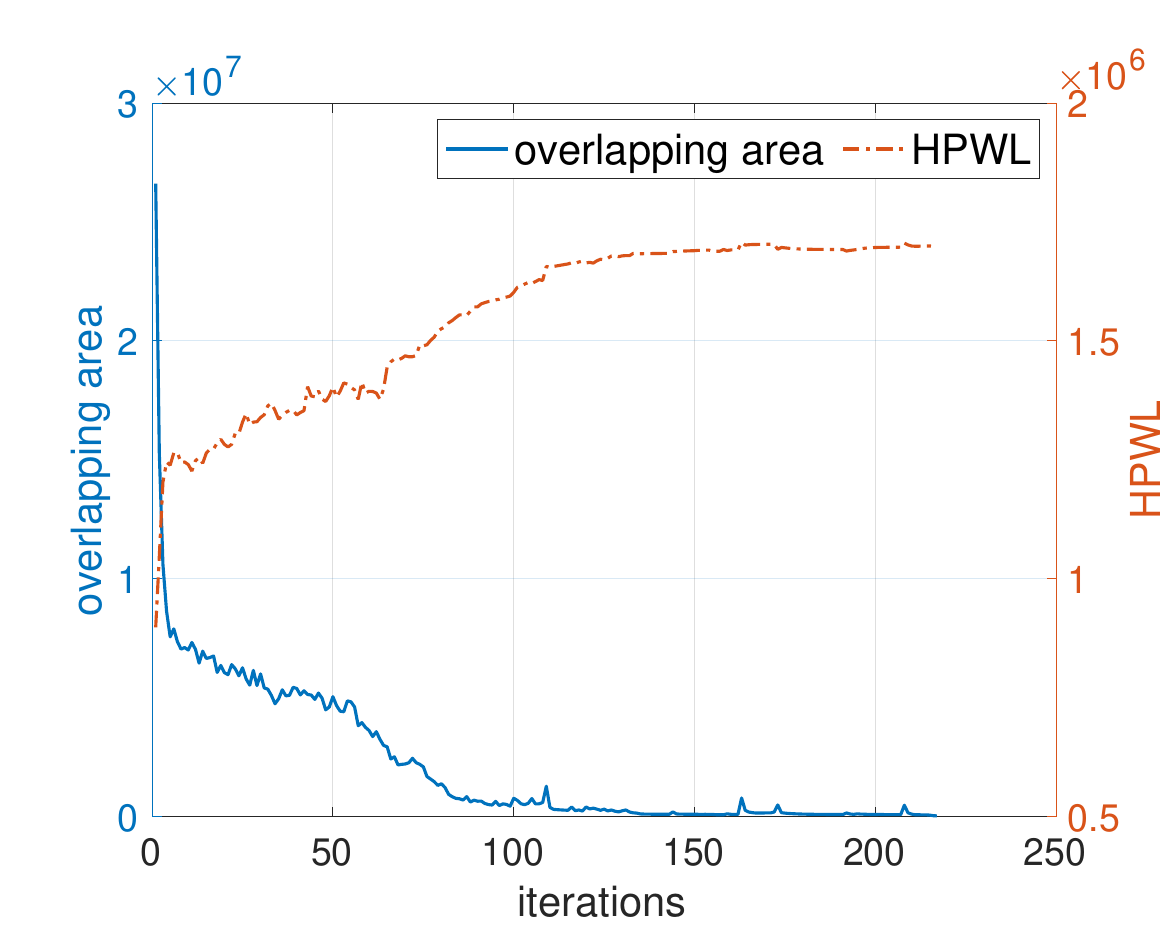}
        \caption{RMAP}
    \end{subfigure}
    \hfill
    \begin{subfigure}[t]{0.32\textwidth}
        \centering
        \includegraphics[width=\textwidth]{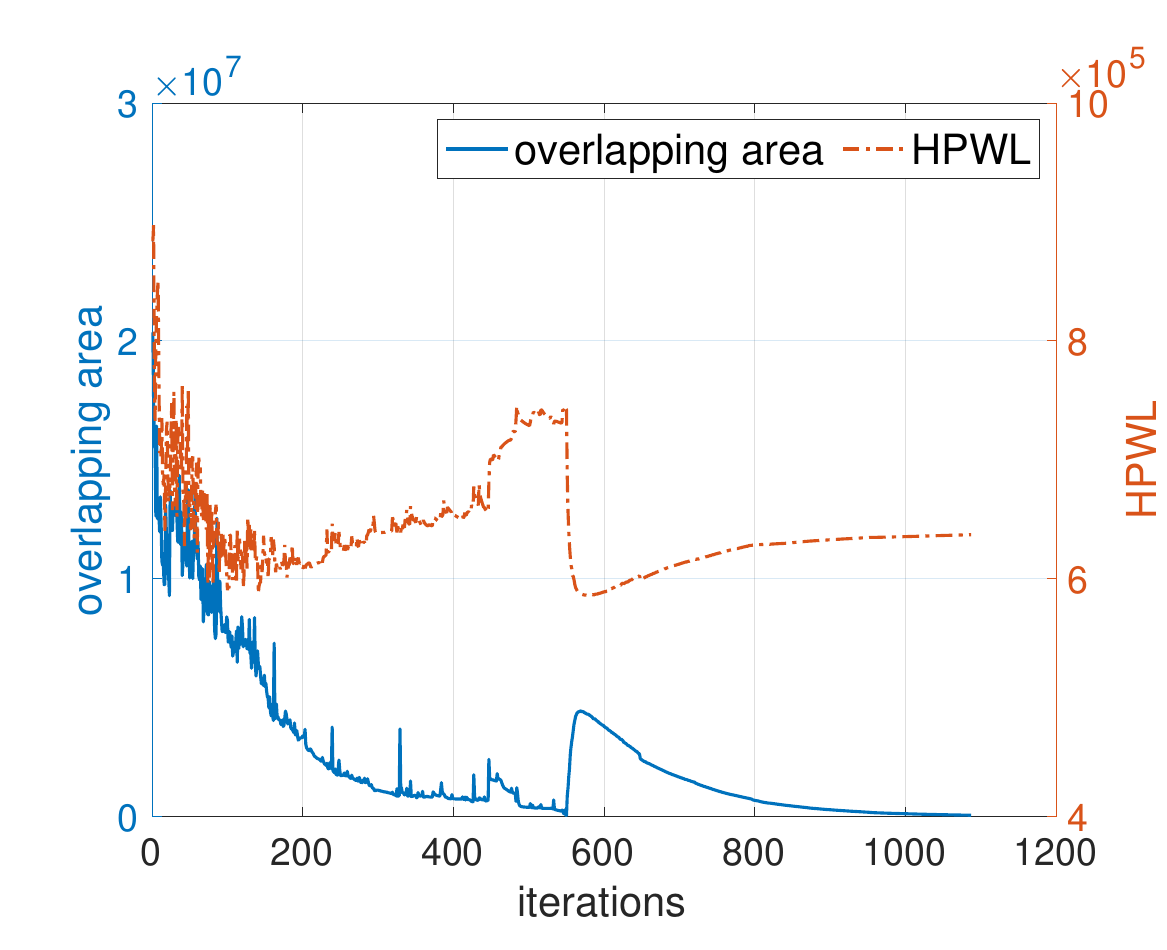}
        \caption{Per-RMAP}
    \end{subfigure}
    \caption{Variation of overlapping area and HPWL over iterations on \texttt{ami49}.}
    \label{fig:MAP-RMAP-PerRMAP}
\end{figure*}

\section{Conclusion and Future Work}\label{sec:conclusion}
We model the fixed-outline floorplanning
as an FSP. However, the conventional
MAP for FSP cannot always obtain legal floorplans because the constraints sets of the floorplanning problem are not convex. We analyze the union convex property of the constraints sets in floorplanning and prove the local convergence of MAP. Besides, we propose the resettable method of alternating projection (RMAP) to improve its initial convergence behavior. Furthermore, a superiorized version, Per-RMAP, is designed to reduce the total wirelength. The experiments show that Per-RMAP achieves nearly optimal results with a speedup of 166$\times$ with a merely 5$\%$ improvement in HPWL compared to the B\&B method. After considering the I/O assignment, Per-RMAP achieves a 6\% decrease in total wirelength. Besides, considering the soft modules case, we obtain 15\% improved efficiency compared with the analytical method PeF. Our experiments show the ability of Per-RMAP to handle larger floorplanning designs while maintaining a good balance between solution quality and efficiency. Our future work is to investigate the capability of handling complex practical constraints, like the ones for 2.5D floorplanning, to further accelerate computations by adopting iterative feasibility-seeking amenable to parallel execution, and modify the algorithm structure to tackle large-scale problems.




\bibliographystyle{IEEEtran}
\bibliography{tcad.bib}

\begin{thebibliography}{10}
\providecommand{\url}[1]{#1}
\csname url@samestyle\endcsname
\providecommand{\newblock}{\relax}
\providecommand{\bibinfo}[2]{#2}
\providecommand{\BIBentrySTDinterwordspacing}{\spaceskip=0pt\relax}
\providecommand{\BIBentryALTinterwordstretchfactor}{4}
\providecommand{\BIBentryALTinterwordspacing}{\spaceskip=\fontdimen2\font plus
\BIBentryALTinterwordstretchfactor\fontdimen3\font minus
  \fontdimen4\font\relax}
\providecommand{\BIBforeignlanguage}[2]{{%
\expandafter\ifx\csname l@#1\endcsname\relax
\typeout{** WARNING: IEEEtran.bst: No hyphenation pattern has been}%
\typeout{** loaded for the language `#1'. Using the pattern for}%
\typeout{** the default language instead.}%
\else
\language=\csname l@#1\endcsname
\fi
#2}}
\providecommand{\BIBdecl}{\relax}
\BIBdecl

\bibitem{yu2023}
S.~Yu, Y.~Censor, M.~Jiang, and G.~Luo, ``{Per-RMAP}: Feasibility-seeking and
  superiorization methods for floorplanning with {I/O} assignment,'' in
  \emph{Proceedings of the 2023 International Symposium of Electronics Design
  Automation ({ISEDA}-2023)}, Nanjing, China, May 8-11 2023, pp. 286--291.

\bibitem{Kahng2011}
A.~B. Kahng, J.~Lienig, I.~L. Markov, and J.~Hu, \emph{VLSI Physical Design:
  From Graph Partitioning to Timing Closure}.\hskip 1em plus 0.5em minus
  0.4em\relax Dordrecht: Springer Netherlands, 2011.

\bibitem{Adya2003}
S.~N. Adya and I.~L. Markov, ``{Fixed-Outline Floorplanning: Enabling
  Hierarchical Design},'' \emph{IEEE Trans. VLSI Syst.}, vol.~11, no.~6, pp.
  1120--1135, 2003.

\bibitem{kahng2018}
A.~B. Kahng, ``Machine learning applications in physical design: Recent results
  and directions,'' in \emph{Proceedings of the 2018 International Symposium on
  Physical Design ({ISPD}-2018), Monterey, CA, USA, March 25-28, 2018}, 2018,
  pp. 68--73.

\bibitem{ma21}
Y.~Ma, L.~Delshadtehrani, C.~Demirkiran, J.~L. Abell{\'{a}}n, and A.~Joshi,
  ``{TAP-2.5D:} {A} thermally-aware chiplet placement methodology for {2.5D}
  systems,'' in \emph{Proceedings of the 2021 Design, Automation {\&} Test in
  Europe ({DATE}-2021)}, Grenoble, France, February 1-5 2021, pp. 1246--1251.

\bibitem{Ahn2011}
B.-G. Ahn, J.~Kim, W.~Li, and J.-W. Chong, ``Effective estimation method of
  routing congestion at floorplan stage for {3D ICs},'' \emph{J. Semicond.
  Technol. Sci.}, vol.~11, no.~4, pp. 344--350, 2011.

\bibitem{Lin2014}
J.~Lin and J.~Wu, ``{F-FM}: Fixed-outline floorplanning methodology for
  mixed-size modules considering voltage-island constraint,'' \emph{IEEE Trans.
  Comput. Aided Des. Integr. Circuits Syst.}, vol.~33, no.~11, pp. 1681--1692,
  2014.

\bibitem{funke2016}
J.~Funke, S.~Hougardy, and J.~Schneider, ``An exact algorithm for wirelength
  optimal placements in {VLSI} design,'' \emph{Integration (Amst.)}, vol.~52,
  pp. 355--366, 2016.

\bibitem{Kahng21}
A.~B. Kahng, ``Advancing placement,'' in \emph{Proceedings of the 2021
  International Symposium on Physical Design ({ISPD}-2021)}, Virtual Event,
  USA, March 22-24 2021, pp. 15--22.

\bibitem{Yair2023}
Y.~Censor, ``Superiorization: The asymmetric roles of feasibility-seeking and
  objective function reduction,'' \emph{Appl. Set-Valued Anal. Optim.}, vol.~5,
  no.~3, pp. 325--346, 2023.

\bibitem{simon1956}
H.~A. Simon, ``Rational choice and the structure of the environment.''
  \emph{Psychol. Rev.}, vol.~63, no.~2, p. 129, 1956.

\bibitem{von1951}
J.~Von~Neumann, \emph{Functional Operators (AM-22), Volume 2: The Geometry of
  Orthogonal Spaces}.\hskip 1em plus 0.5em minus 0.4em\relax Princeton
  University Press, 1951.

\bibitem{Escalante2011}
R.~Escalante and M.~Raydan, \emph{Alternating Projection Methods, Fundamentals
  of Algorithms}.\hskip 1em plus 0.5em minus 0.4em\relax Philadelphia, PA, USA:
  Society for Industrial and Applied Mathematics (SIAM), 2011.

\bibitem{Yair1997}
Y.~Censor and S.~A. Zenios, \emph{Parallel Optimization: Theory, Algorithms,
  and Applications}.\hskip 1em plus 0.5em minus 0.4em\relax New York, NY, USA:
  Oxford University Press, 1997.

\bibitem{Murata1996}
H.~Murata, K.~Fujiyoshi, S.~Nakatake, and Y.~Kajitani, ``{VLSI} module
  placement based on rectangle-packing by the sequence-pair,'' \emph{IEEE
  Trans. Comput. Aided Des. Integr. Circuits Syst.}, vol.~15, no.~12, pp.
  1518--1524, 1996.

\bibitem{Chang2000}
Y.~Chang, Y.~Chang, G.~Wu, and S.~Wu, ``{B\textsuperscript{*}-Trees}: A new
  representation for non-slicing floorplans,'' in \emph{Proceedings of the 37th
  Conference on Design Automation ({DAC}-2000)}, Los Angeles, CA, USA, June 5-9
  2000, pp. 458--463.

\bibitem{Anand2012}
S.~Anand, S.~Saravanasankar, and P.~Subbaraj, ``Customized simulated annealing
  based decision algorithms for combinatorial optimization in {VLSI}
  floorplanning problem,'' \emph{Comput. Optim. Appl.}, vol.~52, no.~3, pp.
  667--689, 2012.

\bibitem{chen2008}
T.-C. Chen, Y.-W. Chang, and S.-C. Lin, ``A new multilevel framework for
  large-scale interconnect-driven floorplanning,'' \emph{IEEE Trans. Comput.
  Aided Des. Integr. Circuits Syst.}, vol.~27, no.~2, pp. 286--294, 2008.

\bibitem{yan2010defer}
J.~Z. Yan and C.~Chu, ``{DeFer}: Deferred decision making enabled fixed-outline
  floorplanning algorithm,'' \emph{IEEE Trans. Comput. Aided Des. Integr.
  Circuits Syst.}, vol.~29, no.~3, pp. 367--381, 2010.

\bibitem{ji2021quasi}
P.~Ji, K.~He, Z.~Wang, Y.~Jin, and J.~Wu, ``A quasi-{Newton}-based floorplanner
  for fixed-outline floorplanning,'' \emph{Comp. Oper. Res.}, vol. 129, p.
  105225, 2021.

\bibitem{Zhan2006}
Y.~Zhan, Y.~Feng, and S.~S. Sapatnekar, ``A fixed-die floorplanning algorithm
  using an analytical approach,'' in \emph{Proceedings of the 2006 Conference
  on Asia South Pacific Design Automation ({ASP-DAC}-2006)}, Yokohama, Japan,
  January 24-27 2006, pp. 771--776.

\bibitem{Li2022}
X.~Li, K.~Peng, F.~Huang, and W.~Zhu, ``{PeF}: Poisson’s equation-based
  large-scale fixed-outline floorplanning,'' \emph{IEEE Trans. Comput. Aided
  Des. Integr. Circuits Syst.}, vol.~42, no.~6, pp. 2002--2015, 2023.

\bibitem{lin2018fast}
J.-M. Lin, T.-T. Chen, Y.-F. Chang, W.-Y. Chang, Y.-T. Shyu, Y.-J. Chang, and
  J.-M. Lu, ``A fast thermal-aware fixed-outline floorplanning methodology
  based on analytical models,'' in \emph{International Conference on
  Computer-Aided Design (ICCAD)}, 2018, pp. 1--8.

\bibitem{Xu2022}
Q.~Xu, H.~Geng, S.~Chen, B.~Yuan, C.~Zhuo, Y.~Kang, and X.~Wen,
  ``{GoodFloorplan}: Graph convolutional network and reinforcement
  learning-based floorplanning,'' \emph{IEEE Trans. Comput. Aided Des. Integr.
  Circuits Syst.}, vol.~41, no.~10, pp. 3492--3502, 2022.

\bibitem{Liu2022}
Y.~Liu, Z.~Ju, Z.~Li, M.~Dong, H.~Zhou, J.~Wang, F.~Yang, X.~Zeng, and
  L.~Shang, ``Floorplanning with graph attention,'' in \emph{Proceedings of the
  59th Design Automation Conference ({DAC}-2022)}, San Francisco, CA, USA, July
  10 - 14 2022, pp. 1303--1308.

\bibitem{bandeira2020fast}
V.~Bandeira, M.~Foga{\c{c}}a, E.~M. Monteiro, I.~Oliveira, M.~Woo, and R.~Reis,
  ``Fast and scalable {I/O} pin assignment with divide-and-conquer and
  hungarian matching,'' in \emph{IEEE International New Circuits and Systems
  Conference (NEWCAS)}, 2020, pp. 74--77.

\bibitem{Kim2012}
M.-C. Kim, D.-J. Lee, and I.~L. Markov, ``{SimPL}: An effective placement
  algorithm,'' \emph{IEEE Trans. Comput. Aided Des. Integr. Circuits Syst.},
  vol.~31, no.~1, pp. 50--60, 2012.

\bibitem{Viswanathan2005}
N.~Viswanathan and C.-C. Chu, ``{FastPlace}: Efficient analytical placement
  using cell shifting, iterative local refinement,and a hybrid net model,''
  \emph{IEEE Trans. Comput. Aided Des. Integr. Circuits Syst.}, vol.~24, no.~5,
  pp. 722--733, 2005.

\bibitem{westra2005towards}
J.~Westra and P.~Groeneveld, ``Towards integration of quadratic placement and
  pin assignment,'' in \emph{{IEEE} Computer Society Annual Symposium on
  {VLSI}: New Frontiers in {VLSI} Design ({ISVLSI})}, Tampa, FL, {USA}, May
  11-12 2005, pp. 284--286.

\bibitem{ajayi2019openroad}
T.~Ajayi and D.~Blaauw, ``{OpenROAD}: Toward a self-driving, open-source
  digital layout implementation tool chain,'' in \emph{Proceedings of
  Government Microcircuit Applications and Critical Technology Conference},
  2019.

\bibitem{Bauschke2011}
H.~H. Bauschke and P.~L. Combettes, \emph{Convex Analysis and Monotone Operator
  Theory in Hilbert Spaces}, ser. {CMS} Books in Mathematics.\hskip 1em plus
  0.5em minus 0.4em\relax Berlin: 2nd ed., Springer, 2017.

\bibitem{Gubin1967}
L.~G. Gubin, B.~T. Polyak, and E.~V. Raik, ``The method of projections for
  finding the common point of convex sets,'' \emph{{USSR} Comput. Math. Math.
  Phys.}, vol.~7, no.~6, pp. 1--24, 1967.

\bibitem{Bauschke2024}
H.~H. Bauschke and W.~M. Moursi, ``On the {Douglas-Rachford} algorithm for
  solving possibly inconsistent optimization problems,'' \emph{Mathematics of
  Operations Research}, vol.~49, no.~1, pp. 58--77, 2024.

\bibitem{Bauschke2014}
H.~H. Bauschke and D.~Noll, ``On the local convergence of the
  {Douglas–Rachford} algorithm,'' \emph{Arch. Math.}, vol. 102, no.~6, pp.
  589--600, 2014.

\bibitem{Artacho2016}
F.~J.~A. Artacho, J.~M. Borwein, and M.~K. Tam, ``Global behavior of the
  {Douglas-Rachford} method for a nonconvex feasibility problem,'' \emph{J.
  Glob. Optim.}, vol.~65, no.~2, pp. 309--327, 2016.

\bibitem{chretien1996cyclic}
S.~Chr\'{e}tien and P.~Bondon, ``Cyclic projection methods on a class of
  nonconvex sets,'' \emph{Numer. Funct. Anal. Optim.}, vol.~17, no. 1-2, pp.
  37--56, 1996.

\bibitem{chretien2020}
S.~Chr{\'e}tien and P.~Bondon, ``Projection methods for uniformly convex
  expandable sets,'' \emph{Mathematics}, vol.~8, no.~7, p. 1108, 2020.

\bibitem{yair2020}
B.~Schultze, Y.~Censor, P.~Karbasi, K.~E. Schubert, and R.~W. Schulte, ``An
  improved method of total variation superiorization applied to reconstruction
  in proton computed tomography,'' \emph{{IEEE} Trans. Med. Imaging}, vol.~39,
  no.~2, pp. 294--307, 2020.

\bibitem{stable}
\BIBentryALTinterwordspacing
M.~Lemmon, ``Chapter 5: Stability in the sense of {Lyapunov},'' in
  \emph{Nonlinear Control Systems (EE 60580)}.\hskip 1em plus 0.5em minus
  0.4em\relax University of Notre Dame, 2017. [Online]. Available:
  \url{https://www3.nd.edu/~lemmon/courses/ee580/}
\BIBentrySTDinterwordspacing

\bibitem{Elsner1992}
L.~Elsner, I.~Koltracht, and M.~Neumann, ``Convergence of sequential and
  asynchronous nonlinear paracontractions,'' \emph{Numer. Math.}, vol.~62,
  no.~1, pp. 305--319, 1992.

\bibitem{he2011system}
L.~He, S.~Elassaad, Y.~Shi, Y.~Hu, and W.~Yao, ``System-in-package: electrical
  and layout perspectives,'' \emph{Foundations and Trends{\textregistered} in
  Electronic Design Automation}, vol.~4, no.~4, pp. 223--306, 2011.

\bibitem{censor2017}
Y.~Censor, ``Can linear superiorization be useful for linear optimization
  problems?'' \emph{Inverse Probl.}, vol.~33, no.~4, p. 044006, 2017.

\bibitem{mcnc}
\BIBentryALTinterwordspacing
K.~Ko{\'{z}}mi{\'{n}}ski, ``Benchmarks for layout synthesis---evolution and
  current status,'' in \emph{Design Automation Conference (DAC)}, 1991, pp.
  265--270. [Online]. Available:
  \url{http://vlsicad.eecs.umich.edu/BK/MCNCbench}
\BIBentrySTDinterwordspacing

\bibitem{gsrc}
\BIBentryALTinterwordspacing
W.~Dai, L.~Wu, and S.~Zhang, ``Floorplanning slot: {GSRC} floorplan benchmark
  suite,'' MARCO GSRC T2 Fabric: Bookshelf, 2001. [Online]. Available:
  \url{http://vlsicad.eecs.umich.edu/BK/GSRCbench}
\BIBentrySTDinterwordspacing

\bibitem{Liu2008}
W.~Liu and A.~Nannarelli, ``Net balanced floorplanning based on elastic energy
  model,'' in \emph{Nordic Microelectronics Event (NORCHIP)}, 2008, pp.
  258--263.

\bibitem{chen2006}
T.-C. Chen and Y.-W. Chang, ``{Modern floorplanning based on
  B\textsuperscript{*}-tree and fast simulated annealing},'' \emph{IEEE Trans.
  Comput. Aided Des. Integr. Circuits Syst.}, vol.~25, no.~4, pp. 637--650,
  2006.

\bibitem{Samaranayake2009}
M.~Samaranayake, H.~Ji, and J.~Ainscough, ``Development of a force directed
  module placement tool,'' in \emph{2009 Ph. D. Research in Microelectronics
  and Electronics}, 2009, pp. 152--155.

\end{thebibliography}


\section{Appendix A: Examples to Illustrate the Limitations of MAP}\label{sec:example-n5}
Example~\ref{eg:n4} in Sec.~\ref{sec:MAP-property-and-challenges} is a representative example that demonstrates the occurrence of oscillations.

Examples~\ref{eg:n3} and \ref{eg:n3v} highlight the challenges of handling oscillations, showing that variations in scanning order and relaxation parameters may not effectively address the problem, especially when the initialization is poor.

\begin{eg}{\bf{\texttt{n3}: 3-module case, oscillation occurs, regardless of scanning order, at a bad initialization.}}\label{eg:n3}
Place three modules with widths $w=(3,4,5)$ and heights $h=(3,4,5)$ into a region with size $(W,H)=(11,11)$. Initialize modules at $z^0=(0,2,6,4,2,0)$ as shown in Fig.~\ref{fig:3-cell-oscillation}(a) and fix the relaxation parameters $\lambda_n$ at 1.

However, even with various scanning orders, oscillation still occur. In one sweep, the scanning order can be determined using the following orders:
\begin{enumerate}
    \item \textbf{Area order:} Scan by the area of modules in non-ascending order, i.e., repeatedly apply $P_{12}P_{13}P_{23}$ in this case. After one sweep, the sequence $\left\{z^n\right\}$ reach the state shown in Fig.~\ref{fig:3-cell-oscillation}(b), where it remains.
    \item \textbf{Position order:} Scan by the positions of modules in non-descending order, i.e., repeatedly apply $P_{23}P_{13}P_{12}$ in this case. The sequence $\left\{z^n\right\}$ stays unchanged at the solution depicted in Fig.~\ref{fig:3-cell-oscillation}(a).  
    \item \textbf{Random order:} Order is generated randomly, then the sequence oscillates between the positions depicted in Fig.~\ref{fig:3-cell-oscillation}(a) and Fig.~\ref{fig:3-cell-oscillation}(b) in random order.
\end{enumerate}

\begin{figure}[ht]
	\centering
	\begin{minipage}{0.3\linewidth}
		\centerline{\includegraphics[width=\textwidth]{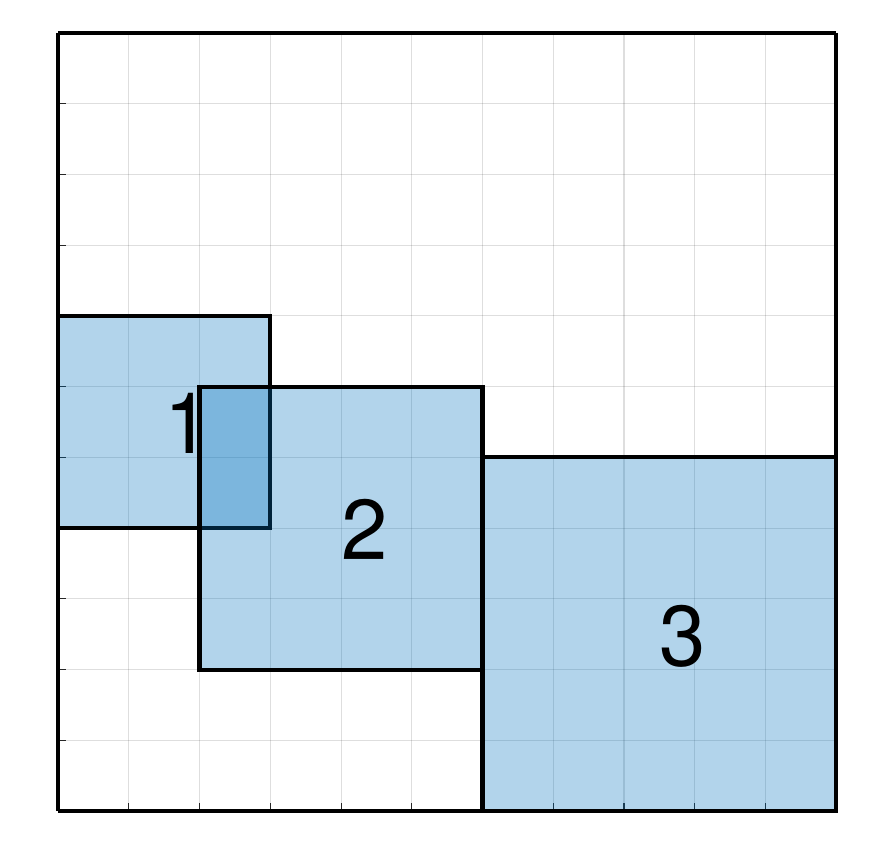}}
		\centerline{(a)}
	\end{minipage}
	\hspace{20pt}
	\begin{minipage}{0.3\linewidth}
		\centerline{\includegraphics[width=\textwidth]{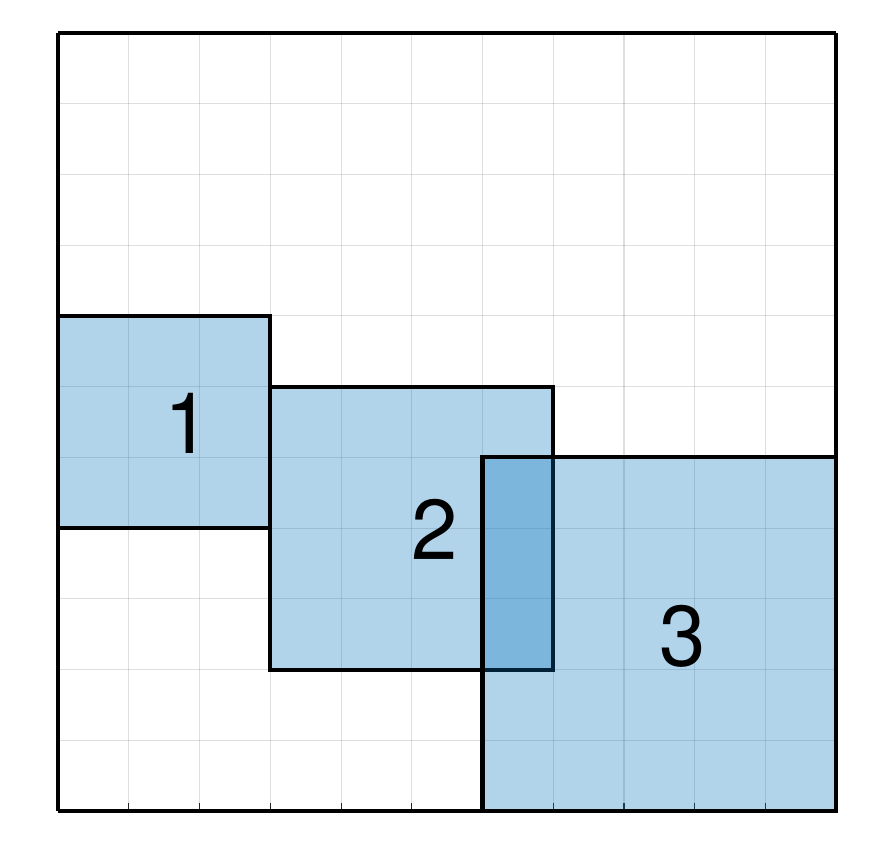}}
		\centerline{(b)}
	\end{minipage}
\caption{Oscillation for \texttt{n3}}
\label{fig:3-cell-oscillation}
\end{figure} 
\end{eg}

\begin{eg}
{\bf{\texttt{n3v}: 3-module case, oscillation occurs regardless of constant relaxation parameter at a bad initialization.}}\label{eg:n3v}
Place three modules with widths $w=(2,2,2)$ and heights $h=(3,4,5)$ within a region of size $(W,H)=(5,11)$. The initial modules positions are set at $z^0=(0,1,3,2,1,0)$ as shown in Fig.~\ref{fig:3-cell-oscillation-varying-lambda}(a). The projection is performed using a constant relaxation parameter $\lambda \in (0,2]$ and the modules are scanned in area order. We consider $\lambda$ in this range as \cite{Yair1997} has proven convergence under the constraints of convex sets. 
\begin{figure}[ht]
\centering
\includegraphics[scale=0.3]{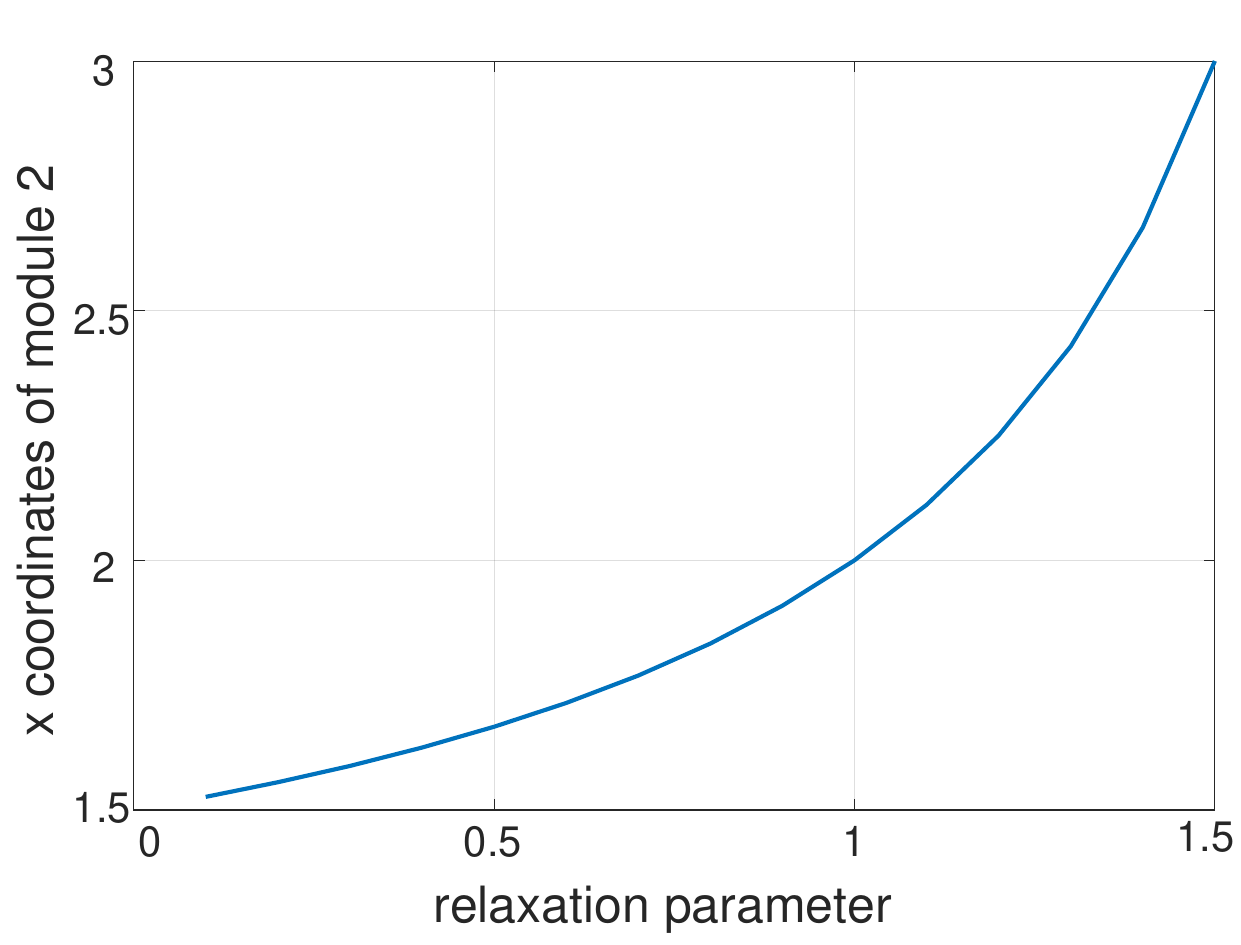}
\caption{Stuck phenomena for \texttt{n3v}: Get stuck at $z=(0,x_2(\lambda),3,2,1,0)$, where $x_2(\lambda)$ denotes the relationship between $x$-coordinate of module 2 and constant relaxation parameter $\lambda \in (0,1.5]$.}%
\label{fig:change-of-x-cell2}%
\end{figure}

Fig.~\ref{fig:change-of-x-cell2} illustrates the relationship between $x_2$ and constant $\lambda \in (0,1.5]$. 
Fig.~\ref{fig:3-cell-oscillation-varying-lambda} shows the initialization ($\lambda = 0$), the stuck position when $\lambda =0.1$, and the stuck position when $\lambda =1.5$.

\begin{figure}[t]
\centering
	\begin{minipage}{0.12\linewidth}
		\centerline{\includegraphics[width=\textwidth]{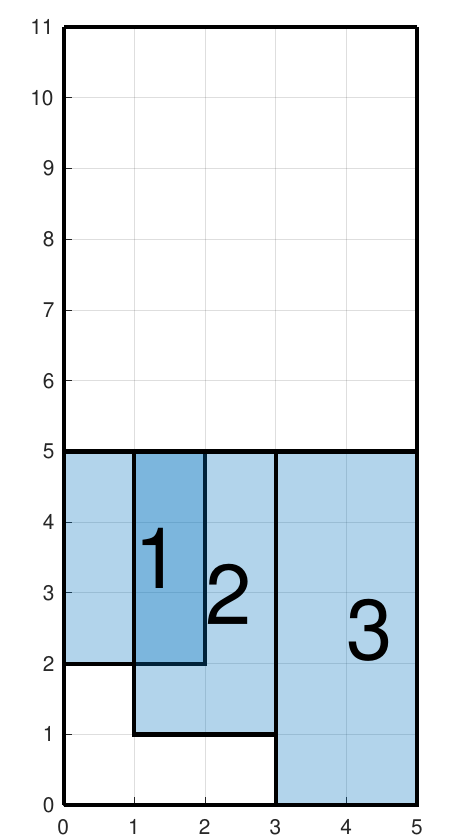}}
		\centerline{(a) $\lambda$ = 0}
	\end{minipage}
	\hspace{10pt}
	\begin{minipage}{0.12\linewidth}
		\centerline{\includegraphics[width=\textwidth]{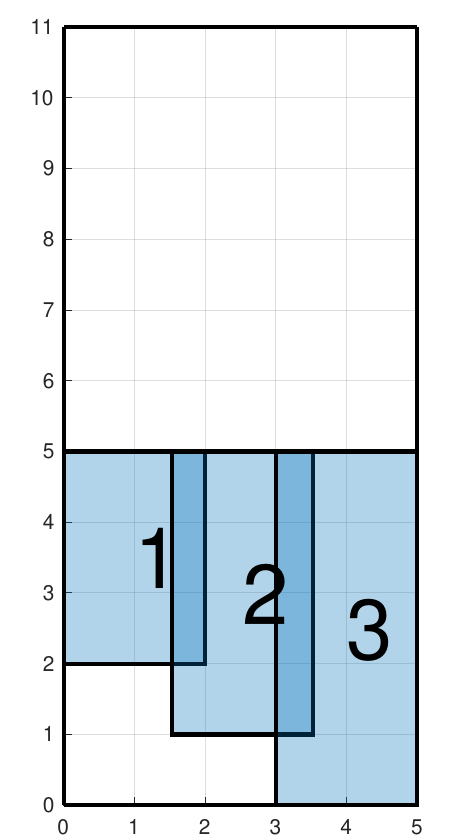}}
		\centerline{(b) $\lambda = 0.1$}
	\end{minipage}
	\hspace{10pt}
	\begin{minipage}{0.12\linewidth}
		\centerline{\includegraphics[width=\textwidth]{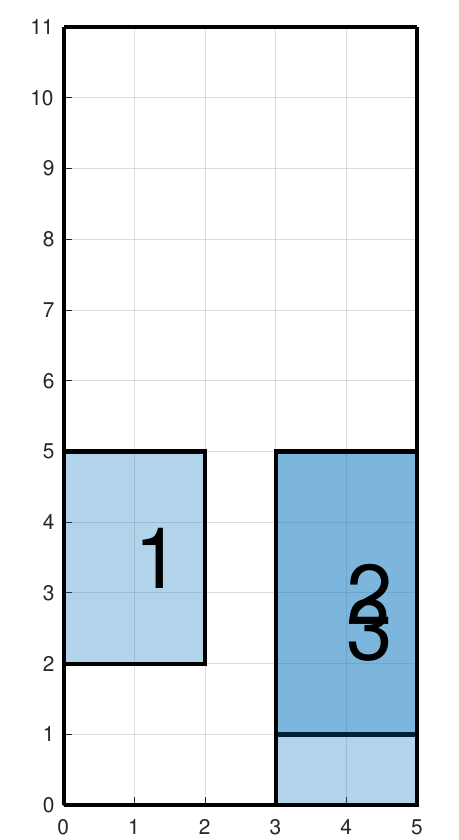}}	 
		\centerline{(c) $\lambda = 1.5$}
	\end{minipage}
	\caption{Stuck phenomena for \texttt{n3v} under different $\lambda$.}
	\label{fig:3-cell-oscillation-varying-lambda}
\end{figure}

For $\lambda \in (1.5, 2]$, MAP exhibits oscillation among four positions, as shown in Fig.~\ref{fig:oscillation-lambda-1.5-2} due to the following reasons:
\begin{enumerate}
    \item Only the $x$-coordinates change: With a maximum overlap of 2 in the $x$-direction and a minimum overlap of 3 in the $y$-direction, favoring the reduction of overlap in the $x$-direction for every pair of modules.
    \item Modules are placed adjacent to the boundary: Modules are bounded and with $\lambda > 1.5$, projections remove at least 3 units of overlap.
    \item Oscillations: It occurs when two modules have an overlap and share the same widths. The effort required to project onto $C_{i,j,\mathsf{L}}$ and $C_{i,j,\mathsf{R}}$, is the same, leading to oscillation.
\end{enumerate}
\begin{figure}[!t]
\centering
\includegraphics[scale=0.2]{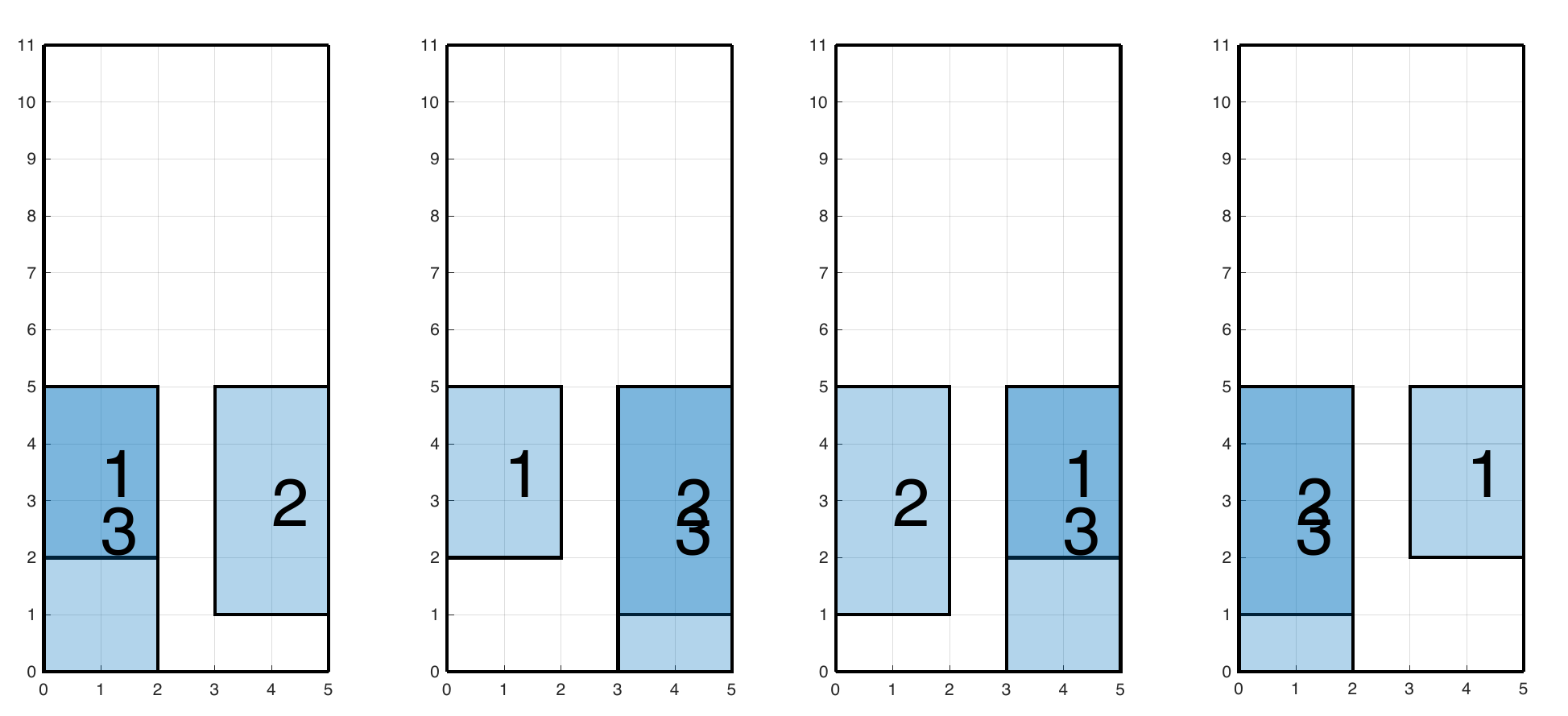} 
\caption{Oscillation for \texttt{n3v} when $\lambda \in (1.5,2]$.}
\label{fig:oscillation-lambda-1.5-2}
\end{figure}
\end{eg}

Furthermore, consider an extreme example \texttt{n5} with zero whitespace. It emphasizes the sensitivity of the problem, where slight variations in initialization can lead to a wide range of explored positions and fail to reach a feasible solution.

\begin{eg}{\bf{\texttt{n5}: 5-module case oscillation.}}\label{eg:n5}
Place five modules with widths $w = (1,2,1,2,1)$ and heights $h = (1,1,2,1,2)$ into a region with size $(W,H)=(3,3)$.
Initialize modules at $z^0=(2,1,1,1,0,1,2,0,0,1)$, fix the relaxation parameter at 1 and scan by position order. The sequence oscillates among 12 positions, as shown in Fig.~\ref{fig:5-cell-oscillations}.

However, if initialized at $\tilde{z}^0=(1,1,1.5,0.5,0, 1,2,0,0,1)$, MAP could find a feasible solution with only a single iteration.
\begin{figure}[t]
\centering
\includegraphics[scale=0.6]{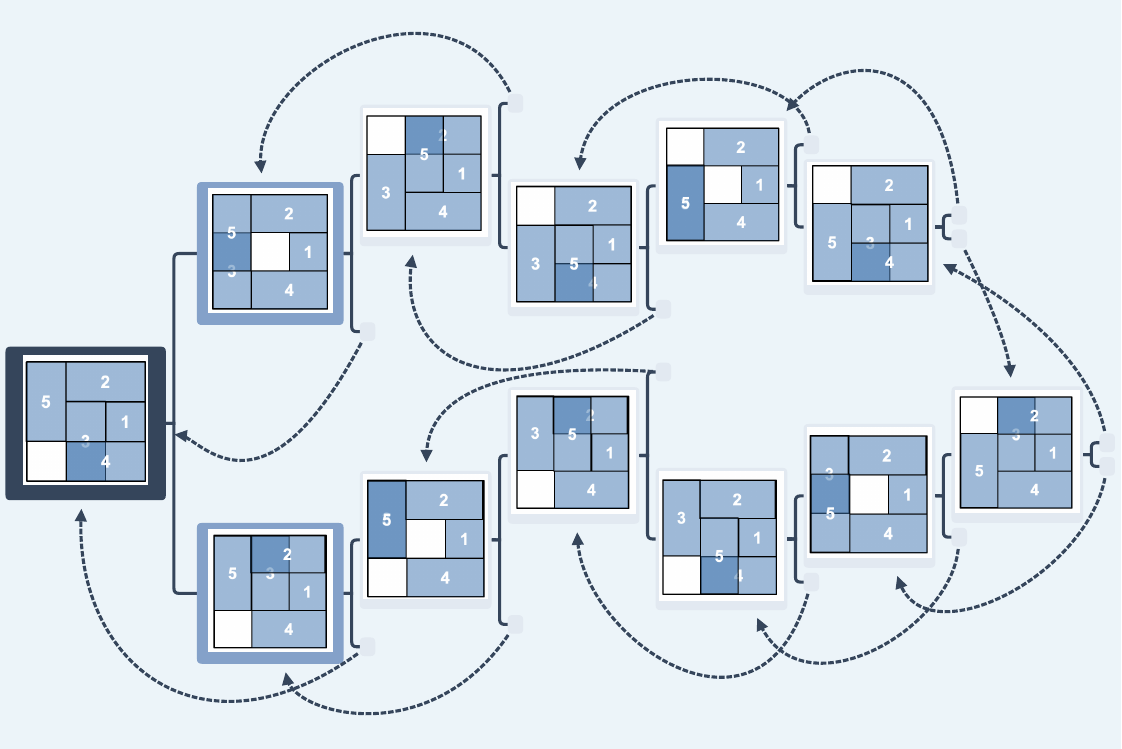} 
\caption{Oscillations among 12 positions.}%
\label{fig:5-cell-oscillations}%
\end{figure}
\end{eg}

\section{Appendix B: More on Local Convergence of MAP}\label{sec:more-local-convergence}

This appendix derives Lemma~\ref{lemma:active2} to replace Lemma~\ref{lemma:active}.

\begin{lemma}
For any $z^* \in F(\mathcal{T}_{t})$, the number of active indices of constraint $C_t$ at $z^*$ is at most two, i.e., $|K_t(z^*)| \in \{1,2\}$.
\end{lemma}

\begin{proof}
Given $z^* \in F(\mathcal{T}_{t})$, we have $z^* \in \mathcal{T}_{t}(z^*)$, $z^* \in \mathcal{P}_{t}(z^*)$, $z^* \in C_{t}$, and $d(z^*,C_{t})=0$.

Since $C_{t,\mathsf{L}}\cap C_{t,\mathsf{R}} = \emptyset$, we cannot have both $d(z^*,C_{t,\mathsf{L}})=0$ and $d(z^*,C_{t,\mathsf{R}})=0$. Thus, either $P_{t,\mathsf{L}}(z^*) \nsubseteq P_{t}(z^*)$ or $P_{t,\mathsf{R}}(z^*) \nsubseteq P_{t}(z^*)$, i.e., either $\mathsf{L} \notin K_t(z^*)$ or $\mathsf{R} \notin K_t(z^*)$. Similarly, either $\mathsf{B} \notin K_t(z^*)$ or $\mathsf{A} \notin K_t(z^*)$. Therefore, $|K_t(z^*)| \le 2$.
\end{proof}

\begin{lemma}\label{lemma:k1}
Given $C_t$ and $z^*\in F(\mathcal{T}_t)$, there exists $\epsilon > 0$ such that every $z \in B(z^*,\epsilon)$ has $K_t(z) = K_t(z^*)$ when $|K_t(z^*)|=1$. Furthermore, $\epsilon = r_t$ below satisfies this condition:
\begin{align}\label{eq:k1}
r_t = \min\{d_{sep}(z^*,C_t),\ \frac{d_{sep}(z^*,C_t)+d_{esc}(z^*,C_t)}{2}\}.
\end{align}
\end{lemma}

\begin{proof}
Let $\{k\} = K_{t}(z^*)$ and $k' \in \left\{\mathsf{L},\mathsf{R},\mathsf{B},\mathsf{A}\right\}/\{k\}$.

When there is no ambiguity, we denote  $d_{sep}(z^*,C_t)$ and $d_{sep}(z^*,C_t)$ as $d_{sep}$ and $d_{esc}$, respectively, in this proof.

If $d_{sep} \le d_{esc}$, every $z \in B(z^*, d_{sep})$ has $K_t(z) = k$, since $z$ does not move far enough to escape from $C_{t,k}$ and arrive at any other $C_{t,k'}$, i.e., $d(z,C_{t,k}) = 0 < d(z,C_{t,k'})$.

If $d_{sep} > d_{esc}$, every $z \in B(z^*, (d_{sep}+d_{esc})/2)$ has $K_t(z) = k$, since
$d(z,C_{t,k}) < (d_{sep}+d_{esc})/2 - d_{esc} = (d_{sep}-d_{esc})/2$ and $d(z,C_{t,k'}) > d_{sep} - (d_{sep}+d_{esc})/2 = (d_{sep}-d_{esc})/2$.

Therefore, every $z \in B(z^*,r_t)$, with the $r_t$ defined in~\eqref{eq:k1}, has $K_t(z)=K_t(z^*)$ when $|K_t(z^*)|=1$.
\end{proof}

\begin{lemma}\label{lemma:k2}
Given $C_t$ and $z^*\in F(\mathcal{T}_t)$, there exists $\epsilon > 0$ such that every $z \in B(z^*,\epsilon)$ has $K_t(z) \subseteq K_t(z^*)$ when $|K_t(z^*)|=2$. Furthermore, $\epsilon = r_t$ below satisfies this condition:
\begin{align}
\begin{split}
& r_t = \min\{g(x_i^*,x_j^*,w_i,w_j), g(y_i^*,y_j^*,h_i,h_j)\}, \\
\end{split}
\end{align}
where $C_t=C_{i,j}$ and $g(a,b,c,d)=|(a-b)+(c-d)/2|/\sqrt{2}$.
\end{lemma}

\begin{proof}
Assume $C_t$ is the non-overlap constraint between module $m_i$ and $m_j$, i.e., $C_t=C_{i,j}$.
The subsets $C_{t,\mathsf{L}}$ and $C_{t,\mathsf{R}}$ are separated by the hyperplane $\mathcal{V}_{t,\mathsf{L,R}} = \{z\ |\ x_i-x_j + (w_i-w_j)/2=0\}$ with identical distance from these two sets. Similarly, $C_{t,\mathsf{B}}$ and $C_{t,\mathsf{A}}$ are divided by the hyperplane $\mathcal{V}_{t,\mathsf{B,A}} = \{z\ |\ y_i-y_j + (h_i-h_j)/2=0\}$, while maintaining the same distance to each set. 
Since $d(z^*,\mathcal{V}_{t,\mathsf{L,R}})=g(x_i^*,x_j^*,w_i,w_j)$ and $d(z^*,\mathcal{V}_{t,\mathsf{B,A}})=g(y_i^*,y_j^*,h_i,h_j)$, if $r_t = \min\{d(z^*,\mathcal{V}_{t,\mathsf{L,R}}),d(z^*,\mathcal{V}_{t,\mathsf{B,A}})\}$, then every $z \in B(z^*,r_t)$ remains in the same side of the two hyperplanes as $z^*$, so no new active indices would be created. As a result, $K_t(z)\subseteq K_t(z^*)$.
\end{proof}

\begin{lemma}\label{lemma:active2}
Given $C_t$ and $z^*\in F(\mathcal{T}_t)$, every $z \in B(z^*,r_t)$ has $K_t(z) \subseteq K_t(z^*)$, where

\begin{align}
\begin{split}
r_t \! = &\!
\begin{cases}
\min\{d_{sep}(z^*,C_t),\ \frac{d_{sep}(z^*,C_t)+d_{esc}(z^*,C_t)}{2}\},\\
\qquad \qquad \qquad \qquad \qquad \qquad \qquad
\text{ if } \vert K_t(z^*)\vert \! =\! 1,\\
\min\{g(x_i^*,x_j^*,w_i,w_j), g(y_i^*,y_j^*,h_i,h_j)\}, \\
\qquad \qquad \qquad \qquad \qquad \qquad \qquad
\text{ if } \vert K_t(z^*)\vert \! = \! 2,
\end{cases}\\
\text{with } & C_t = C_{i,j}.
\end{split}
\end{align}
\end{lemma}

\begin{proof}
It is straightforward to derive from Lemma~\ref{lemma:k1} and~\ref{lemma:k2}.
\end{proof}

\end{CJK*}
\end{document}